\documentclass[12pt]{article}         

 \usepackage[T1]{fontenc}              
 \usepackage[latin1]{inputenc}        
 \usepackage[width=16cm, height=22cm]{geometry} 
 \input{kobib}  
 \usepackage{mathtools,amssymb}          
\usepackage{hyperref} 
 \usepackage[amsmath,thmmarks,hyperref]{ntheorem} 
 \usepackage{icomma}                   
 \usepackage{enumerate}                
\usepackage{xstring}
 \usepackage{color}                    
 \usepackage{setspace}                 
  \usepackage{needspace}                
 \usepackage{url}                      
 \usepackage{mathrsfs}                 
 \usepackage{graphicx}                 
\usepackage{tikz}                     
 \usepackage{slashed}
 \usepackage[all,cmtip]{xypic}
\usepackage{verbatim}
\usepackage[capitalise]{cleveref}
\usepackage{euscript}

\usepackage[textsize=scriptsize, backgroundcolor=yellow, linecolor=black, textcolor=black]{todonotes}

 \usepackage[normalem]{ulem}
\newcounter{comments}
\newenvironment{displaycomment}{\begin{list}{}{\rightmargin=1cm\leftmargin=1cm}\item\sf\begin{small}}{\end{small}\end{list}}

 \numberwithin{equation}{subsection}
 
 \usepackage[numbers,square]{natbib}
 
  \definecolor{MyBlue}{RGB}{25,106,180}
\def\quot#1{``#1''}

\theoremstyle{nonumberplain}  
\theoremheaderfont{\itshape}  
\theorembodyfont{\normalfont}  
\theoremseparator{.}  
\theoremsymbol{$\Box$}
\newtheorem{proof}{Proof} 
  
\theoremstyle{plain}  
\theoremheaderfont{\normalfont\bfseries}  
\theorembodyfont{\itshape}  
\theoremsymbol{~}
\theoremseparator{.}  
\newtheorem{proposition}{Proposition}[subsection]  
\newtheorem{corollary}[proposition]{Corollary}  
\newtheorem{lemma}[proposition]{Lemma}  
\newtheorem{theorem}[proposition]{Theorem}

\newtheorem{maintheorem}{Theorem}[section]   
\newtheorem{maincorollary}[maintheorem]{Corollary}

\theorembodyfont{\normalfont}  
 
\newtheorem{remark}[proposition]{Remark}
\newtheorem{example}[proposition]{Example}  
  
\newtheorem{definition}[proposition]{Definition}

\def\mathscr#1{\EuScript{#1}}

\theoremstyle{nonumberplain}
\theorembodyfont{\normalfont}  

\DeclareMathOperator*{\Per}{Per}
\DeclareMathOperator*{\sskew}{skew}
\newcommand{\Skew}{\mathrm{Skew}}
\newcommand{\Alt}{\mathrm{Alt}}

\DeclareMathOperator*{\inv}{inv}

\newcommand*{\res}{\operatorname{res}}
\newcommand{\R}{\mathbb{R}}

\newcommand{\N}{\mathbb{N}}
\newcommand{\Z}{\mathbb{Z}}

\newcommand{\dd}{\mathrm{d}}
\newcommand{\pr}{\mathrm{pr}}

\newcommand{\Spin}{\mathrm{Spin}}
\newcommand{\SO}{\mathrm{SO}}

\newcommand{\id}{\mathrm{id}}
\newcommand{\dis}{_{\mathrm{dis}}}

\newcommand{\Hom}{\mathrm{Hom}}

\newcommand{\CExt}{\mathrm{c}\mathscr{E}\mathrm{xt}}
\newcommand{\DCCExt}{\mathrm{dc}\text{-}\mathrm{c}\mathscr{E}\mathrm{xt}}

\def\TwoGroupFusFac#1#2{\mathcal{G}(#1, #2)}
\def\XExt#1{\mathscr{X}\text{-}\mathrm{c}\mathscr{E}\mathrm{xt}(#1)}

\newcommand{\Lietwogroups}{\mathscr{L}\text{ie-2-}\mathscr{G}\text{rp}}
\newcommand{\XMod}{\mathscr{X}\text{-}\mathscr{M}\mathrm{od}}

\newcommand{\U}{\mathrm{U}}

\newcommand{\ev}{\mathrm{ev}}
\newcommand{\const}{\texttt{const}}

\crefname{equation}{\unskip}{\unskip}

\def\BCSS{^{^{_\mathrm{BCSS}}}}
\def\Waldorf{^{^{_\mathrm{W}}}}
\def\diff{^\mathrm{dflg}}
\def\rep{\mathrm{rep}}
\def\si{_{\mathrm{si}}}

\title{Lie 2-groups from loop group extensions}
\author{Matthias Ludewig and Konrad Waldorf}
\date{}

\usepackage{amstext}

\begin{document}

\sloppy
\maketitle

\begin{abstract}
\noindent
We give a very simple construction of the string 2-group as a strict Fr\'echet Lie 2-group.
The corresponding crossed module is defined using the conjugation action of the loop group on its central extension, which drastically simplifies several constructions previously given in the literature.
More generally, we construct strict 2-group extensions for a Lie group from a central extension of its based loop group, under the assumption that this central extension is disjoint commutative.
 We show in particular that this condition is automatic in the case that the Lie group is semisimple and simply connected.
\end{abstract}

\tableofcontents

\section{Introduction}

\numberwithin{equation}{section}

In the seminal paper \cite{BCSS} by Baez, Crans, Stevenson and Schreiber, a certain Fr\'echet Lie 2-group extension of a  Lie group $G$ of Cartan type (i.e., compact, connected, simple, simply connected) was constructed, using a particular presentation of the universal central extension of the loop group $LG$. For $G=\Spin(n)$, their construction realizes a model for the string 2-group. 

In an attempt to generalize this construction, the second-named author described in \cite{WaldorfStringGroupModels,WaldorfString, WaldorfTransgressive}  another, diffeological 2-group extension of an arbitrary Lie group $G$, using an arbitrary central extension of $LG$ equipped with a certain additional structure -- a multiplicative fusion product. 
If $G$ is of Cartan type, such a central extension can be provided canonically, and one can prove abstractly that the corresponding 2-group is weakly equivalent to the one of Baez et al.

The purpose of the present paper is to (drastically) simplify and to unify both constructions. 
For this purpose, we study in the first part of this paper, \cref{section-central-extensions}, central extensions of loop groups and of groups of paths, in the category of Fr\'echet Lie groups. We identify a  property of central extensions of a loop group, \emph{disjoint commutativity}, as  crucial for the construction of 2-groups. A central extension 
\begin{equation*}
1\to \U(1) \to \widetilde{LG} \to LG \to 1
\end{equation*}  
is disjoint commutative if elements $\Phi, \Phi'\in\widetilde{LG}$ commute if they project to loops $\gamma, \gamma'\in LG$ with disjoint supports. Disjoint commutativity has been introduced in \cite{WaldorfTransgressive} as a property of transgressive central extension, and it is relevant for the theory of nets of operator algebras \cite{Gabbiani1993}. 
Our first result is the following (see \cref{CorollarySemisimpleSimplyConnectedDisjointCommutative} and, for a more general statement, \cref{ThmModifiedProduct}). 

\begin{maintheorem}
If $G$ is semisimple and simply connected, then all central extensions of $LG$ are disjoint commutative.
\end{maintheorem}

The relevance of disjoint commutativity for Lie 2-groups lies in the construction of crossed module actions. 
We denote by $P_eG$ the Fr\'echet Lie group of paths in $G$ that start at the identity element $e$, and all whose derivatives at both end points vanish. 
We denote by $\widetilde{\Omega_{(0,\pi)} G}$ the restriction of $\widetilde{L G}$ to the group $\Omega_{(0,\pi)}G$ of those loops whose support is in their first half $(0,\pi) \subset S^1$. 
Then, we consider the Lie group homomorphism
\begin{equation*}
t : \widetilde{\Omega_{(0,\pi)} G} \to P_eG
\end{equation*}  
thats projects to the first half of the base loop, considered as a (closed) path. In order to turn the homomorphism $t$ into a crossed module, it remains to provide a  crossed module action $\alpha$ of $P_eG$ on $\widetilde{\Omega_{(0,\pi)} G}$.
In the above-mentioned paper \cite{BCSS} by Baez et al., such an action is constructed (in a slightly different setting) using Lie-algebraic methods and particularities of a specific model of $\widetilde{\Omega G}$. 
In the second above-mentioned approach  \cite{WaldorfStringGroupModels,WaldorfTransgressive}, a crossed module action is constructed using the given fusion product.

In our setting, the required  action $\alpha$ is both simple and canonical: a path $\gamma\in P_eG$ is first ,,doubled``  to a thin loop in $\Omega G$, lifted to $\widetilde{\Omega G}$, and then acts by conjugation on  $\Phi \in \widetilde{\Omega_{(0,\pi)} G}$, see \cref{SectionCrossedModulesFromCExt}. In general, this canonical action $\alpha$ will not be a crossed module action, as it does not satisfy the so-called Peiffer identity. One of our main insights is that this problem is resolved when $\widetilde{\Omega G}$ is disjoint commutative, see \cref{ThmUniqueAction}.

\begin{maintheorem}
\label{main-one}
If $\widetilde{\Omega G}$ is a disjoint commutative central extension of $\Omega G$, then  the canonical action $\alpha$ turns $t: \widetilde{\Omega_{(0, \pi)} G} \to P_e G$ into a central  crossed module. 
Moreover, if $G$ is semisimple, then $\alpha$  is the only such action. 
\end{maintheorem}

We emphasize that \cref{main-one} provides a drastic simplification of the construction of 2-group extensions; in particular, for the construction of string 2-group models. Neither additional structure on the central extension is needed, nor any other special knowledge about its concrete model.

\medskip

We denote by $X(\widetilde{\Omega G})$ the crossed module of \cref{main-one}, and now consider the special case where $G$ is of Cartan type, and $\widetilde{\Omega G}$ has level $k\in \Z$. We denote by $X\BCSS(G, k)$ the  crossed module constructed by Baez et al.\ at the same level. In \cref{Comparison-with-BCSS} we construct a canonical, strict homomorphism 
\begin{equation}
\label{comparison-baez}
X(\widetilde{\Omega G}) \to X\BCSS(G, k)\text{,}
\end{equation}
of crossed modules of Fr\'echet Lie groups. 
On the other hand, we consider a disjoint commutative central extension $\widetilde{LG}$ with a fusion product $\lambda$, and  denote by $X\Waldorf(\widetilde{\Omega G}, \lambda)$ the diffeological crossed module corresponding to the diffeological 2-group of  \cite{WaldorfStringGroupModels,WaldorfString, WaldorfTransgressive}. 
Under the canonical inclusion  of Fr\'echet manifolds into diffeological spaces, we construct in \cref{diffeological-model} another, strict homomorphism
\begin{equation*}
X\Waldorf(\widetilde{\Omega G}, \lambda) \to X(\widetilde{\Omega G})\text{.}
\end{equation*}  
\cref{equivalence-to-BCSS,weak-equivalence-diffeological} prove the following. 

\begin{maintheorem}
\label{equivalence-baez-waldorf}
The homomorphisms \cref{comparison-baez} and \cref{comparison-baez} establish weak equivalences
\begin{equation*}
X\Waldorf(\widetilde{\Omega G}, \lambda) \cong X(\widetilde{\Omega G})
\quad\text{ and }\quad
X(\widetilde{\Omega G}) \cong X\BCSS(G, k)\text{.}
\end{equation*}
\end{maintheorem}

In particular, this shows that the two earlier constructions $X\Waldorf(\widetilde{\Omega G}, \lambda)$ and $X\BCSS(G, k)$ are canonically and strictly isomorphic, a fact that is very difficult to observe when only looking at these two 2-groups.   

\bigskip

Another aspect we investigate in this paper concerns the 2-groups associated to the crossed modules discussed above. As these two structures (2-groups and crossed modules)  are canonically equivalent, our crossed module $X(\widetilde{\Omega G})$ determines a Fr\'echet Lie 2-group $\mathcal{G}(X(\widetilde{\Omega G}))$, whose group of objects is $P_eG$, and whose group of morphisms is the semi-direct product
\begin{equation}
\label{SemiDirectProductIntro}
\widetilde{\Omega_{(0,\pi)}G} \rtimes_{\alpha} P_eG\text{.}
\end{equation}
The Fr\'echet Lie 2-group of Baez et al.\ has a similar structure. 
However, the diffeological construction of the second-named author results into a Lie 2-group whose group of morphisms is a subgroup of $\widetilde{\Omega G}$, and is hence ,,nicer``.  
It turns out that the missing ingredient to identify the semi-direct product \cref{SemiDirectProductIntro} with a subgroup of $\widetilde{\Omega G}$ is a homomorphism
\begin{equation*}
  i: P_e G \longrightarrow \widetilde{\Omega G}
\end{equation*}
such that $i(\gamma) \in \widetilde{\Omega G}$ lies over the thin loop corresponding to the path $\gamma$. 
Such a map was called \emph{fusion factorization} in \cite{KristelWaldorf1}.
Here we have the following result, see \cref{ThmExistenceFusionFactorization}.

\begin{maintheorem}
If $G$ is semisimple, fusion factorizations are unique. If $G$ is additionally simply connected, then fusion factorizations exist.
\end{maintheorem}

In particular, if $G$ is simply connected and semisimple, then every central extension of $\Omega G$ carries a unique fusion factorization.
The whole situation can be summarized as follows.

\begin{maincorollary}
If $G$ is simply connected and semisimple, then for every central extension $\widetilde{\Omega G}$ of $\Omega G$, there exists a unique central crossed module $X(\widetilde{\Omega G})$ of Fr\'echet Lie groups with underlying homomorphism 
\begin{equation*}
t : \widetilde{\Omega_{(0,\pi)}G} \to P_eG.
\end{equation*} Moreover, there exists a unique Lie 2-group  $\TwoGroupFusFac{\widetilde{\Omega G}}{i}$, with objects and morphisms
\begin{equation*}
\xymatrix{\widetilde{P_e G^{[2]}} \ar@<4pt>[r]^-{s} \ar@<-2pt>[r]_-{t} & P_eG\text{,}}
\end{equation*} 
where $\widetilde{P_e G^{[2]}} \subset \widetilde{\Omega G}$ is the subgroup over those loops that are flat at  $0$ and $\pi$.
Finally, $X(\widetilde{\Omega G})$ and   $\TwoGroupFusFac{\widetilde{\Omega G}}{i}$ correspond to each other under the adjunction between crossed modules and 2-groups. 
\end{maincorollary}

Finally, we come back to the main motivation of the whole topic, the construction of models for the string 2-group. In \cite{BCSS} it was proved that
the geometric realization of the Lie 2-group corresponding to the crossed module $X\BCSS(G, k)$ is a 3-connected cover of $G$, which -- for $G=\Spin(d)$ -- is the defining property of a string 2-group. In \cref{classification-of-Lie-2-groups} we generalize this result slightly from Lie groups of Cartan type to arbitrary simple and simply connected Lie groups. \Cref{th:classification} states the following.

\begin{maintheorem}
If $G$ is simple and simply connected,  and $\widetilde{\Omega G}$ is a basic central extension, then the geometric realization of the (canonically isomorphic) Lie 2-groups $\mathcal{G}(X(\widetilde{\Omega G}))$ and $\TwoGroupFusFac{\widetilde{\Omega G}}{i}$ are 3-connected covers of $G$. In particular, if $G=\Spin (d)$, both are models for the string 2-group.  
\end{maintheorem}

\paragraph{Acknowledgements.} 
ML gratefully acknowledges support from SFB 1085 ``Higher invariants'' funded by the German Research Foundation (DFG).

\numberwithin{equation}{subsection}

\section{Loop groups and their central extensions}
\label{section-central-extensions}
       
In this section we recall some relevant results about central extensions of loop groups and path groups, and also add a couple of new results which we will use later. 
In particular, in \cref{disjoint commutativity}, we discuss and investigate  the relatively new notion of disjoint commutative central extensions.    
       
\subsection{Path groups and loop groups}

Throughout, let $G$ be a connected (finite-dimensional) Lie group.
We denote the identity element by $e$, and we denote by $LG = C^\infty(S^1, G)$ the smooth loop group of $G$. 
We always identify $S^1 = \R/2\pi \Z$.
For $I\subset S^1$, we write 
\begin{equation}
\label{NotationLIG}
L_I G = \{\gamma \in LG \mid \gamma(t) = e ~\text{whenever}~t \notin I\}.
\end{equation}
We say that a map $f: M \to N$ between manifolds is \emph{flat} at a point $p \in M$, if all directional  derivatives of $f$ vanish at all orders at the point $p$. 
We observe, in particular, that all elements of $L_{(a, b)} G$ are flat at $t=a, b$ (unless $(a, b) = (0, 2\pi)$).
We also denote by $\Omega G \subset LG$ the subset of loops $\gamma$ that are flat at $t=0$ and satisfy $\gamma(0)= e$.
Analogously to the notation above, we also write 
\begin{equation*}
\Omega_I G = L_IG \cap \Omega G.
\end{equation*}

We denote by $PG$ the space of all smooth maps $\gamma: [0, \pi] \to G$ that are flat at their endpoints, and by $P_e G \subset PG$ the subset of paths $\gamma$ with $\gamma(0) = e$.
We then have a short exact  sequence
\begin{equation*}
\xymatrix{
  \Omega_{(0, \pi)}G \ar[r] & P_e G \ar[r]^-{\mathrm{ev}} & G,
}
\end{equation*}
where the first map is the restriction of $\gamma \in \Omega_{(0, \pi)} G$ to the interval $[0, \pi]$, and the second map is the endpoint evaluation.
For two paths $\gamma_1$, $\gamma_2$ with a common initial point and a common end point, we define a loop $\gamma_1\cup\gamma_2 \in LG$ by
\begin{equation} \label{CupMap}
  (\gamma_1 \cup \gamma_2)(t) := \begin{cases} \gamma_1(t) & t \in [0, \pi] \\ \gamma_2(2\pi - t) & t \in [\pi, 2\pi] \end{cases}.
\end{equation}
We identify the fibre product $P_e G^{[2]} = P_e G \times_G P_e G$ with its image in $LG$ under this map.

For non-trivial $G$, all loop groups and path groups discussed above are infinite-dimensional Lie groups, which are modeled on nuclear Fr\'echet spaces.
Their Lie algebras are obtained by taking the appropriate path space inside the Lie algebra $\mathfrak{g}$ of $G$.

\begin{remark} \label{RemarkRegular}
The Fr\'echet Lie groups $L_IG$, $\Omega G$, and $P_e G$ are \emph{regular} in the sense of \cite[Def.~3.12]{NeebCentral}, which means that every smooth curve in their Lie algebra can be integrated to a smooth curve in the group.
This follows from the fact that such an integral can be calculated pointwise in the loop parameter, which gives a smooth curve in $G$. 
Then, as solutions to ordinary differential equations depend smoothly on the initial data, these curves yield a smooth curve in the appropriate path group.
\end{remark}


It is a corollary of \cite[Prop.~3.4.1]{PressleySegal} that if $G$ is semisimple, there are no non-trivial Lie group homomorphisms from $LG$ to any abelian Lie group $A$, i.e., every Lie group homomorphism $\varphi: LG \to A$ is $\varphi=1$.
 The following generalization will be  key to the present paper.

\begin{theorem} \label{TheoremSemisimpleNoHoms}
If $G$ is a semisimple Lie group, then the Fr\'echet Lie group $P_e G$ does not admit non-trivial Lie group homomorphisms to any abelian Lie group $A$.
The same is true for the identity components of $\Omega G$ and $L_I G$, for any $I \subseteq S^1$.
\end{theorem}

We need the following lemma.
    
    \begin{lemma} \label{FlatPathLemma}
    For every smooth function $f: [0, a] \to \R$ that is flat at zero, there are smooth functions $g_1, g_2 : [0, a] \to \R$ that are also flat at zero and satisfy $f(t) = g_1(t) g_2(t)$ for all $t \in [0, a]$.
    \end{lemma}
    
For the proof of \cref{FlatPathLemma}, we need the following observation: consider the following property for a map $f: [0, a] \to \R$ with $f(0)=0$. \begin{enumerate}
 \item[($\star$)] {\itshape $f$ is smooth on $(0, a]$ and for each $n \in \N$, there exists $\varepsilon >0$ such that $|f(t)| \leq t^n$ for each $t \in [0, \varepsilon]$. }
\end{enumerate}
An easy exercise shows that $f$ satisfies ($\star$) if and only if  $f$ is smooth on $[0,a]$ and flat at zero. 

 \begin{proof}[of \cref{FlatPathLemma}]
    By ($\star$) we may choose, for each $n \in \N$, an $\varepsilon_n > 0$ such that $|f(t)| \leq t^n$ for each $t \in [0, \varepsilon_n]$. 
  We choose these numbers such that the sequence $\varepsilon_1, \varepsilon_2, \dots$ is strictly decreasing and converges to zero.
    For each $n \in \N$, we choose smooth functions $h_n :[\varepsilon_{n+1}, \varepsilon_n] \to \R$ such that $\frac{1}{2}t^n \leq h_n(t) \leq t^n$ for all $t \in [\varepsilon_{n+1}, \varepsilon_n]$, in such a way that the functions $h_{n+1}$ and $h_n$ fit smoothly together.
    As  $(\varepsilon_n)_{n\in \N}$ forms a null sequence, there is a smooth function $h$ on $(0, \varepsilon_1]$ such that $h$ agrees with $h_n$ when restricted to $[\varepsilon_{n+1}, \varepsilon_n]$.
    Setting $h(0) = 0$,  we obtain a function $h$ which by construction satisfies $h(t) \geq \frac{1}{2}|f(t)|$ for each $t \in (0, \varepsilon_1]$, and which satisfies ($\star$), hence is flat at zero.
    We smoothly extend $h$ to a function defined on all of $[0, a]$.
    
    We now set $g_1(t) = f(t)/h(t)^{1/2}$, $g_2(t) = h(t)^{1/2}$.
    It is clear that the function $g_2$ satisfies ($\star$), and so does $g_1$, as $|g_1(t)| \leq h(t) / h(t)^{1/2} = h(t)^{1/2}$.
    Hence both $g_1$ and $g_2$ are flat at zero, and we have $f(t) = g_1(t) g_2(t)$, as required.
  \end{proof}

\begin{proof}[of \cref{TheoremSemisimpleNoHoms}]
We prove the result for $P_e G$, the proof for $L_I G$ is similar.
As $G$ is semisimple, we have  $\mathfrak{g} = [\mathfrak{g}, \mathfrak{g}]$, that is, every element of $\mathfrak{g}$ is linear combination of commutators.
We first show that the same is true for the Lie algebra $P_0 \mathfrak{g}$ of $P_e G$.
Let $x_1, \dots, x_n$ be a vector space basis for $\mathfrak{g}$ and choose numbers $a^{ij}_k \in \R$ with
    \begin{equation*}
       x_k = \sum_{ij=1}^n a^{ij}_k [x_i, x_j].
    \end{equation*}
    Write
    \begin{equation*}
      X(t) = \sum_{k=1}^n f_k(t) x_k
    \end{equation*}
    By \cref{FlatPathLemma}, there exist $g_k, h_k \in P_0\R$ such that $f_k(t) = g_k(t)h_k(t)$, $k=1, \dots, n$.
    Then 
    \begin{equation*}
      X(t)  = \sum_{ijk=1}^n f_k(t) a_k^{ij} [x_i, x_j] = \sum_{ijk=1}^n  a_k^{ij} [g_k(t) x_i, h_k(t) x_j].
    \end{equation*}
    As commutators in $P_0 \mathfrak{g}$ are taken pointwise, this witnesses $X$ as a sum of commutators in the Lie algebra $P_0\mathfrak{g}$.
    
Let now $\varphi: P_eG \to A$ be a Lie group homomorphism with induced Lie algebra homomorphism $\varphi_* : P_0\mathfrak{g} \to \mathfrak{a}$, where $\mathfrak{a}$ is the Lie algebra of $A$.
As $\varphi_*$ sends commutators to commutators, it must send the commutator subspace of $P_0\mathfrak{g}$ to the commutator subspace of $\mathfrak{a}$, which is zero as $A$ (and consequently $\mathfrak{a}$) is abelian.
However, as $[P_0\mathfrak{g}, P_0\mathfrak{g}] = P_0\mathfrak{g}$, this implies that $\varphi_*$ is identically zero.
Since $P_e G$ is regular (see \cref{RemarkRegular}), this implies, together with the fact that $P_eG$ is connected that $\varphi$ itself is trivial; see \cite[Lemma~7.1]{MilnorInfiniteLie}.
\end{proof}

\subsection{Classification of central extensions of loop groups}

We recall that a \emph{central extension} of a (possibly infinite-dimensional, Fr\'echet) Lie group $H$ (by the group $\U(1)$) is a  sequence
\begin{equation*}
1\to\U(1) \longrightarrow \widetilde{H} \stackrel{\pi}{\longrightarrow} H\to 1
\end{equation*}
of Lie groups and Lie group homomorphisms such that it is exact as a sequence of groups, and $\widetilde{H}$ is a principal $\U(1)$-bundle over $H$.
For such a central extension, we always identify $\U(1)$ with its image in $\widetilde{H}$.
A Lie group isomorphism $f : \widetilde{H} \to \widetilde{H}^\prime$ is an \emph{isomorphism of central extensions} if it is base point-preserving and trivial on $\U(1) \subset \widetilde{H}$.
We denote by $\CExt(H)$ the groupoid of central extensions of $H$.

Given two central extension $\widetilde{H}$ and $\widetilde{H}^\prime$, their tensor product  $\widetilde{H} \otimes \widetilde{H}^\prime$ (as $\U(1)$-principal bundles) has a group structure turning it into another central extension.
This defines a symmetric monoidal structure on $\CExt(H)$.
Given a central extension $\widetilde{H}$, the dual circle bundle $\widetilde{H}^*$ has an obvious group structure turning it into a central extension that is inverse to $\widetilde{H}$ with respect to the tensor product.
%
%
Hence, the set $h_0(\CExt(H))$ of isomorphism classes in $\CExt(H)$ is a group.

\medskip

We discuss the classification of central extensions $\widetilde{H}$ for a given (Fr\'echet) Lie group $H$.
Choosing a linear section of the Lie algebra homomorphism $\widetilde{\mathfrak{h}} \to \mathfrak{h}$  induced by the  projection $\widetilde{H} \to H$ gives an identification $\widetilde{h} \cong \mathfrak{h} \oplus \R$, under which the bracket attains the form
\begin{equation*}
  [(X, \lambda), (Y, \mu)] = ([X, Y], \omega(X, Y)),
\end{equation*}
for a continuous Lie algebra 2-cocycle  $\omega$ on $\mathfrak{h}$.
The cocycles corresponding to two different choices of splittings differ by a coboundary; hence, there is a well-defined class in the continuous Lie algebra cohomology group $H^2_c(\mathfrak{h}, \R)$ defined by the central extension $\widetilde{H}$.
This establishes a group homomorphism
\begin{equation} \label{CanonicalMap}
  h_0(\CExt(H)) \longrightarrow H^2_c(\mathfrak{h}, \R).
\end{equation}
This homomorphism is neither injective or surjective in general.
However, if $H$ is simply connected, then  \cref{CanonicalMap} is injective \cite[Thm.~7.12]{NeebCentral} and its image is the subgroup represented by cocycles $\omega$ whose group of periods
\begin{equation*}
  \Per \omega := \Bigl\{ \int_Z \overline{\omega} ~\Bigl|~ Z ~\text{a smooth 2-cycle on}~H\Bigr\} \subseteq \R
\end{equation*}
is contained in $2\pi\Z$. 
Here $\overline{\omega}$ denotes the left invariant 2-form on $G$ determined by $\omega$ \cite{NeebNote}.

In general, if $H$ satisfies  $\pi_1(H)=0$ but is not necessarily connected, we obtain a functor 
\begin{equation*}
\CExt(\pi_0(H)) \to \CExt(H),
\end{equation*} 
given by pullback of central extensions along the group homomorphism $H \to \pi_0(H)$. 
On isomorphism classes, this gives a sequence
\begin{equation}
\label{ClassifyingSequence}
 h_0(\CExt(\pi_0(H))) \longrightarrow h_0(\CExt(H)) \longrightarrow H^2_c(\mathfrak{h}, \R),
\end{equation}
which is exact in the middle if $\pi_1(H) = 0$. 
Indeed, if $\widetilde H$ is a central extension with vanishing cohomology class, then its restriction to the identity component $H_0 \subset H$ still has vanishing cohomology class. But over $H_0$ the map  \cref{CanonicalMap} is injective, showing the restriction of $\widetilde H$ to $H_0$ must be trivial. 
But this implies that $\tilde{H}$ comes from a central extension of $\pi_0(H)$.

\begin{example}
\label{ExampleMickelssonExtension}
For connected and simply connected groups $H$ (where the map \cref{CanonicalMap} is injective), there is an explicit description of the central extension corresponding a 2-cocycle $\omega$ on $\mathfrak{h}$ with $\Per\omega \subset 2\pi \Z$, see \cite[\S4.4]{PressleySegal}. 
Let $\overline{\omega}$ be the left invariant 2-form on $H$ determined by $\omega$.
For a loop $\gamma \in \Omega H$, we define
\begin{equation*}
  C(\gamma) := \exp\left( i \int_{\hat{h}} \overline{\omega}\right),
\end{equation*}
where $h : [0, 1] \to \Omega H$ is a smooth null homotopy of $\gamma$ and $\hat{h} : [0, 1] \times S^1 \to H$ is the corresponding surface in $H$ (such a null homotopy exists as $H$ is simply connected).
By the assumption on $\omega$, the integral of $\overline{\omega}$ over any \emph{closed} surface lies in $2 \pi \Z$, which implies that $C(\gamma)$ is independent of the choice of $h$.
One then defines
\begin{equation*}
  \widetilde{H} = P_{e} H \times \U(1) \bigl/ \sim,
\end{equation*}
where $(\gamma_1, z_1) \sim (\gamma_2, z_2)$ if $\gamma_1(\pi)=\gamma_2(\pi)$ and $C(\gamma_1 \cup \gamma_2) = z_2/z_1$.
The bundle projection $\pi : \widetilde{H} \to H$ is given by $(\gamma, z) \mapsto \gamma(\pi)$ and the group product is
\begin{equation*}
  [\gamma_1, z_1] \cdot [\gamma_2, z_2] = [(\const_{\gamma_1(\pi)} \cdot \gamma_2) * \gamma_1, z_1 z_2],
\end{equation*}
where $*$ denotes concatenation of paths.
This gives a central extension whose image under the map \cref{CanonicalMap} is the cocycle $\omega$ (Propositions~4.4.2 \& 4.5.6 of \cite{PressleySegal}).
\end{example}

\medskip

In the following we consider central extensions of the loop groups $LG$ and $\Omega G$ of a connected Lie group $G$.
As $\pi_k(\Omega G) = \pi_{k+1}(G)$ and
\begin{equation} \label{HomotopyGroupsLG}
  \pi_k(LG) = \pi_k(\Omega G) \oplus \pi_k(G) = \pi_{k+1}(G) \oplus \pi_k(G), 
\end{equation}
it follows that $LG$ and $\Omega G$ are connected if and only if $G$ is simply connected. 
Moreover, since $\pi_2(G) = 0$ for any (finite-dimensional) Lie group $G$, it follows that always $\pi_1(\Omega G)=0$ while $\pi_1(LG) = \pi_1(G)$.
Thus, if $G$ is simply connected, then the map \cref{CanonicalMap} is injective, hence any central extension of $H = LG$ or $\Omega G$ is determined by its corresponding Lie algebra cocycle $\omega$.

\begin{lemma}
\label{lemma:equivariantcocycle}
If $G$ is semisimple, then every 2-cocycle on $L\mathfrak{g}$ and $\Omega \mathfrak{g}$ is cohomologous to a cocycle of the form
\begin{equation} \label{FormulaForCocycle}
  \omega(X, Y) = \int_{S^1} b( X(t), Y^\prime(t)) \dd t
\end{equation}
for a $G$-invariant symmetric bilinear form $b$ on $\mathfrak{g}$.
\end{lemma}

\begin{proof}
It is well known that every $G$-invariant 2-cocycle is of the form \cref{FormulaForCocycle}, see e.g., \cite[Prop.~4.2.4]{PressleySegal}.
For not necessarily $G$-invariant 2-cocycles, the result follows from the general results of \cite{NeebWagemann};
see in particular Example~7.2. 
There, cocycles are decomposed as $f_1 + f_2$, which are necessarily \emph{uncoupled} in the authors terminology, as $\mathfrak{g}$ is semisimple. 
It is not hard to figure out that $f_2$ is necessarily a coboundary and $f_1$ gives a cocycle of the form \cref{FormulaForCocycle}.
\end{proof}

\begin{remark}
If $G$ is compact, simply connected and simple, then there is an isomorphism $H^2_c(L \mathfrak{g}, \R) \cong H^2_c(\Omega \mathfrak{g}, \R) \cong H^3(G, \R)$  that sends the subgroup of classes defining a central extension of $LG$, i.e., the image of \cref{CanonicalMap},  onto the subgroup $H^3(G, \Z)$.
\end{remark}

\begin{remark}
Consider the central extension $\widetilde{LG}$ constructed in \cref{ExampleMickelssonExtension} from a 2-cocycle $\omega$ on $L\mathfrak{g}$. If $\omega$ is of the form \cref{FormulaForCocycle} for a bilinear form $b$ on $\mathfrak{g}$, then $\widetilde{LG}$ can be equivalently described as follows.
The elements of $\widetilde{LG}$ can be represented by pairs $(\sigma, z)$, where $\sigma: D^2 \to G$ is a smooth map and where $(\sigma_1, z_1) \sim (\sigma_2, z_2)$ if and only if
  \begin{equation*}
      \frac{z_2}{z_1} = \exp\left(2 \pi i\int_{\Sigma} \overline{\nu}\right).
  \end{equation*}
  Here $\Sigma : D^3 \to G$ is a map whose restriction $\Sigma_{|\partial D^3}$ is given by $\sigma_1$ and $\sigma_2$ on its two hemispheres, and $\overline{\nu}$ is the left invariant 3-form on $G$ associated to the Lie algebra cocycle $\nu(x, y, z) = b([x, y], z)$ on $\mathfrak{g}$.
The group structure is realized with the Mickelsson product, see Theorem 6.4.1 of \cite{brylinski1} and \cite{mickelsson1}.
The projection $\widetilde{LG} \to LG$ is given by sending $(\sigma, z) \mapsto \sigma|_{\partial D^2}$, identifying $\partial D^2 = S^1$.
 This description is equivalent to the one  of \cref{FormulaForCocycle} as the transgression of  $\overline{\nu}$ is cohomologous  to  $\overline{\omega}/2\pi$, see \cite[Prop. 4.4.4]{PressleySegal}.
\end{remark}

\subsection{Restrictions of central extensions}

In the following, we assume -- as before -- that $G$ is a connected (finite-dimensional) Lie group.

\begin{lemma} \label{LemmaAutoCExt}
  If $G$ is semisimple, then the automorphism group of a central extension of $LG$, $\Omega G$, $L_I G$ or $\Omega_IG$, for $I \subseteq S^1$, is canonically isomorphic to $\Hom(\pi_1(G), \U(1))$. 
  In particular, if $G$ is semisimple and simply connected, then the categories $\CExt(LG)$, $\CExt(\Omega G)$ and $\CExt(L_IG)$ have only trivial automorphism groups.
\end{lemma}

\begin{proof}
We prove the result for $LG$, the proof for $\Omega G$ and $L_IG$ is similar.
By \cref{HomotopyGroupsLG}, we have $\pi_0(LG) = \pi_1(G)$, hence any group homomorphism $\varphi : \pi_1(G) \to \U(1)$ gives rise to an automorphism $f$ of $\widetilde{LG}$ by setting $f(\Phi) = \varphi([\pi(\Phi)]) \Phi$.
%
%

Conversely, let $f$ be an automorphism of a central extension $\widetilde{LG}$. 
We define a map $\varphi : LG \to \widetilde{LG}$ by
\begin{equation*}
   \varphi(\gamma) = f(\tilde{\gamma}) \tilde{\gamma}^{-1},
\end{equation*}
where $\tilde{\gamma} \in \widetilde{LG}$ is any lift of $\gamma$. We 
observe that $\varphi$ is well-defined and satisfies $f(\Phi) = \varphi(\pi(\Phi)) \Phi$ for all $\Phi \in \widetilde{LG}$.
$\varphi$ is smooth, since $\pi: \widetilde{LG} \to LG$ has smooth local sections.
Then, since $f$ is base-point preserving, we have 
\begin{equation*}
\pi (\varphi(\tilde{\gamma})) = \pi(f(\tilde{\gamma})) \pi(\tilde{\gamma})^{-1}= \gamma \gamma^{-1} = \const_e, 
\end{equation*}
hence $\varphi(\gamma) \in \U(1) \subset \widetilde{LG}$.
The resulting map $\varphi : LG \to \U(1)$ is a group homomorphism, as
\begin{equation*}
  \varphi(\gamma\eta) = f(\tilde{\gamma})f(\tilde{\eta}) \tilde{\eta}^{-1} \tilde{\gamma}^{-1} = f(\tilde{\gamma})\varphi(\eta) \tilde{\gamma}^{-1} = f(\tilde{\gamma}) \tilde{\gamma}^{-1} \varphi(\eta) = \varphi(\gamma)\varphi(\eta),
\end{equation*}
where we used that $\U(1) \subset \widetilde{LG}$ is central. 
If $G$ is semisimple, \cref{TheoremSemisimpleNoHoms} shows that $\varphi$ is trivial on the identity component $(LG)_0$.
This implies that $\varphi$  factors through $\pi_0(LG) = \pi_1(G)$.
\end{proof}

For our construction of 2-group extensions,  central extensions of $\Omega G$ will be relevant. 
On the other hand, central extensions of $LG$ frequently occur in practice.
Therefore, we shall study the relation between the two types of central extensions.
Clearly, restriction of from $LG$ to $\Omega G$ provides a functor
\begin{equation}
\label{FunctorCExtLGtoCExtOmegaG}
  \CExt(LG) \longrightarrow \CExt(\Omega G).
\end{equation}

\begin{lemma}
\label{LemmaEquivalenceLGOmegaG}
  If $G$ is semisimple and simply connected, then the functor \cref{FunctorCExtLGtoCExtOmegaG} is an equivalence.
\end{lemma}

\begin{proof}
By \cref{LemmaAutoCExt}, we only have to check that the functor is a bijection on isomorphism classes.
To this end, consider the commutative diagram
\begin{align}
\label{LGOmegaGLg}
\xymatrix{
 h_0(\CExt(LG)) \ar[r] \ar[d] & h_0(\CExt(\Omega G) \ar[d]\\
 H^2_c(L\mathfrak{g}, \R) \ar[r] & H^2_c(\Omega \mathfrak{g}, \R).
}
\end{align}
where the top horizontal map is induced by the functor \eqref{FunctorCExtLGtoCExtOmegaG}, the bottom horizontal map is pullback along the Lie algebra homomorphism $\Omega \mathfrak{g} \to L\mathfrak{g}$ and the vertical maps are the canonical map \eqref{ClassifyingSequence} for $H= \Omega G$ and $H= L G$, respectively.
That $G$ is simply connected implies that both vertical maps are injective.
On the other hand, as $G$ is semisimple, \cref{lemma:equivariantcocycle} implies that the bottom map is an isomorphism (as both consist of classes determined by cocycles of the specific form \cref{FormulaForCocycle}, which gives the same classification).
We conclude that the top horizontal map must be injective. 

On the other hand, given a central extension $\widetilde{\Omega G}$ of $\Omega G$, one can construct a central extension $\widetilde{LG}$ of $LG$ such that $\widetilde{LG}|_{\Omega G} = \widetilde{\Omega G}$ in the following way. As $G$ is semisimple, we may assume that $\widetilde{\Omega G}$ is classified by a cocycle $\omega$ of the specific form  \cref{FormulaForCocycle}. Because $G$ is simply-connected, the $\mathrm{Ad}_G$-invariance of $\omega$ integrates to a $G$-action on $\widetilde{\Omega G}$ lifting the conjugation action on $\Omega G$. Identifying $LG = \Omega G \rtimes G$, defining  $\widetilde{LG} := \widetilde{\Omega G} \rtimes G$ gives the claimed central extension.       
\end{proof}

\begin{example}
If $G$ is not simply connected, the conclusion of \cref{LemmaEquivalenceLGOmegaG} is generally false.
Namely, if $\widetilde{G}$ is a finite cover of $G$ (these are defined by elements of $\Hom(\pi_1(G), \U(1))$), then pullback of $\widetilde{G}$ along the evaluation homomorphism $LG \to G$ yields a central extension of $LG$ which is trivial when restricted to $\Omega G$.
\end{example}

If $I \subsetneq (0, 2\pi)$, we can further restrict a central extension of $\Omega G$ along the inclusion $\Omega_I G \subset \Omega G$, which gives functors
\begin{equation}
\label{FurtherRestrictionFunctors}
  \CExt(\Omega G) \longrightarrow \CExt(\Omega_I G).
\end{equation}

\begin{lemma}
\label{LemmaRestrictionIso}
If $G$ is semisimple, then the functor \cref{FurtherRestrictionFunctors} is an equivalence whenever $I \subsetneq (0, 2\pi)$ is connected and non-empty.
\end{lemma}

\begin{proof}
We show that the functor is fully faithful.
To this end, since we are dealing with groupoids, it suffices to show that \cref{FurtherRestrictionFunctors} induces an isomorphism of automorphism groups.
Let $f$ be an automorphism of a central extension $\widetilde{\Omega G}$ of $\Omega G$, inducing an isomorphism $f_I$ of the restricted central extension $\widetilde{\Omega_I G}$.
As $G$ is semisimple, we obtain from (the proof of) \cref{LemmaAutoCExt} that $f(\Phi) = \varphi([\pi(\Phi)]) \Phi$ for some group homomorphism $\varphi: \pi_1(G) \to \U(1)$. 
Now if $f_I$ is trivial, we have $\Phi = f_I(\Phi) = \varphi([\pi(\Phi)])\Phi$ for all $\Phi \in \widetilde{\Omega_I G}$, so $\varphi([\pi(\Phi)]) = 1$.
But this implies that $\varphi$ (hence $f$) is trivial, as any element of $\pi_1(G)$ can be represented by a loop in $\Omega_I G$.
This shows that the induced map on automorphism groups is injective.

Similarly, if $f_I$ is any automorphism of $\widetilde{\Omega_I G}$, then $f_I(\Phi) = \varphi([\pi(\Phi)]) \Phi$ for some $\varphi : \pi_0(\Omega_IG) \to \U(1)$.
But $\pi_0(\Omega_I G) = \pi_0(\Omega G) = \pi_1(G)$, so $f(\Phi) = \varphi([\pi(\Phi)])\Phi$ is an extension of $f_I$ to an automorphism of $\widetilde{\Omega G}$.
This shows that the induced map on automorphism groups is surjective, so the functor is fully faithful. 

It remains to show that the functor is essentially surjective.
Observe that $\Omega_I G = \Omega_{I^\circ} G$, where $I^\circ$ is the interior of $I$, hence we may assume that $I$ is open.
Since $I$ is connected, there exists a diffeomorphism $\varphi : I \to (0, 2\pi)$, which we may choose $\varphi$ to be affine-linear. 
Pre-composition with $\varphi$ induces a group isomorphism $\varphi^* : \Omega G \to \Omega_IG$, which gives rise to an equivalence $\CExt(\Omega_IG) \to \CExt(\Omega G)$. 
Since $G$ is semisimple, any central extension $\widetilde{\Omega G}$ of $\Omega G$ can be represented by a cocycle of the form \cref{FormulaForCocycle}.
It follows that any central extension of $\Omega_I G$ is represented by a cocycle of the form
\begin{equation*}
\begin{aligned}
  \varphi^*\omega(X, Y) &= \int_0^{2\pi} b\bigl((\varphi_*X)(t), (\varphi_*Y)^\prime(t)\bigr) dt \\
  &= \int_0^{2\pi} b\bigl(X(\varphi(t)), Y^\prime(\varphi(t))\varphi^\prime(t) \bigr) dt\\
  &= \int_I b\bigl(X(t), Y^\prime(t)\bigr)dt.
\end{aligned}
\end{equation*}
But this is just the restriction of the cocycle $\omega$.

This shows that if $\widetilde{\Omega_IG}$ is a central extension of $\Omega_IG$, then there exists a central extension $\widetilde{\Omega G}^\prime$ of $\Omega G$ whose  restriction $\widetilde{\Omega_I G}^\prime$ to $\Omega_IG$ is classified by the same Lie algebra cocycle.
Since $\pi_1(\Omega G) = \pi_2(G) = 0$, the sequence \cref{ClassifyingSequence} is exact in the middle, and so $\widetilde{\Omega_I G}^\prime$ and $\widetilde{\Omega_IG}$ differ by a central extension of $\pi_0(\Omega G) = \pi_1(G)$. 
But since the inclusion $\Omega_I G \to \Omega G$ induces an isomorphism on $\pi_0$, we can modify $\widetilde{\Omega G}^\prime$ to achieve $\widetilde{\Omega_I G}^\prime \cong \widetilde{\Omega_IG}$.
Hence, the functor \cref{FurtherRestrictionFunctors} is essentially surjective.
\end{proof}

\begin{example}
Without the assumption of semisimplicity, \cref{LemmaRestrictionIso} is false in general:  
\cref{ExampleNotSemisimple} provides an example of a non-trivial central extension of $\Omega G$ such that the restriction to a suitable $\Omega_I G$ is trivial.
\end{example}




\subsection{Disjoint commutativity}

\label{disjoint commutativity}

Let $G$ be a finite-dimensional, connected Lie group.
It turns out that to construct a 2-group from central extensions of the loop group $LG$, it is important that these central extensions satisfy a certain extra property, \emph{disjoint commutativity}, which was first studied systematically in \cite[\S3.3]{WaldorfTransgressive}.

\begin{definition}[Disjoint commutativity] 
  A central extension $\widetilde{LG}$ of $LG$ is called \emph{disjoint commutative} if for all $I, J \subset S^1$ with $I \cap J = \emptyset$ the subgroups $\widetilde{L_I G}$ and $\widetilde{L_J G}$ of $\widetilde{LG}$ commute.
\end{definition}

The following lemma is crucial.
Recall that a bihomomorphism $b$ on a group $K$ is called \emph{skew} if $b(g, h) = b(h, g)^{-1}$.
Moreover, by an \emph{interval}, we mean an open, nonempty and connected proper subset $I \subset S^1$.

\begin{lemma}
\label{LemmaObstructionBihomomorphism}
 Let $\widetilde{LG}$ be a central extension of $LG$, and suppose that  $G$ is semisimple. Then, there exists a unique bihomomorphism
\begin{equation}
b : \pi_1(G) \times \pi_1(G) \longrightarrow \U(1)
\end{equation}
such that, for disjoint intervals $I,J \subset S^1$ and all $\gamma\in L_IG$, $\eta\in L_JG$, we have 
\begin{equation*}
b([\gamma],[\eta]) =\tilde{\gamma} \tilde{\eta} \tilde{\gamma}^{-1} \tilde{\eta}^{-1} 
\end{equation*}
where $\tilde{\gamma}$ and $\tilde{\eta}$ are arbitrary lifts of $\gamma$, $\eta$ to $\widetilde{LG}$. Moreover, $b$ is skew. 
\end{lemma}

\begin{proof}

For $\Phi \in \widetilde{L_I G}$, $\Psi \in \widetilde{L_JG}$, we have $\pi(\Phi^{-1} \Psi \Phi \Psi^{-1}) = \const_e$, hence the commutator $\Phi^{-1} \Psi \Phi \Psi^{-1}$ is contained in $\U(1) \subset \widetilde{LG}$.
Observe that this commutator only depends on $\pi(\Phi)$ and $\pi(\Psi)$, as replacing $\Phi = z\Phi$ and $\Psi = w \Psi$, $z, w \in \U(1)$, leads to the same result.
Hence we obtain a map
\begin{equation} \label{BihomomorphismB}
  B_{IJ}: {L_IG} \times {L_JG} \longrightarrow \U(1), \qquad (\gamma, \eta) \longmapsto \tilde{\gamma} \tilde{\eta} \tilde{\gamma}^{-1} \tilde{\eta}^{-1}.
\end{equation}
where $\tilde{\gamma}$ and $\tilde{\eta}$ are arbitrary lifts of $\gamma$, $\eta$ to the central extension.
$B_{IJ}$ is smooth as $\widetilde{LG}$ admits smooth local sections. 
We calculate
\begin{equation*}
\begin{aligned}
  B_{IJ}({\gamma}_1, {\eta})B_{IJ}({\gamma}_2, {\eta}) &= (\tilde{\gamma}_1 \tilde{\eta} \tilde{\gamma}_1^{-1} \tilde{\eta}^{-1})(\tilde{\gamma}_2 \tilde{\eta} \tilde{\gamma}^{-1}_2 \tilde{\eta}^{-1})\\
  &= \tilde{\gamma}_1 (\tilde{\gamma}_2 \tilde{\eta} \tilde{\gamma}_2^{-1} \tilde{\eta}^{-1}) \tilde{\eta} \tilde{\gamma}^{-1}_1 \tilde{\eta}^{-1}\\
  &= (\tilde{\gamma}_1 \tilde{\gamma}_2) \tilde{\eta} (\tilde{\gamma}_1 \tilde{\gamma}_2)^{-1} \tilde{\eta}^{-1} \\&= B_{IJ}({\gamma}_1{\gamma}_2, {\eta})
\end{aligned}
\end{equation*}
and 
\begin{equation*}
\begin{aligned}
  B_{IJ}({\gamma}, {\eta}_1)B_{IJ}({\gamma}, {\eta}_2) &= (\tilde{\gamma} \tilde{\eta}_1 \tilde{\gamma}^{-1} \tilde{\eta}_1^{-1})(\tilde{\gamma} \tilde{\eta}_2 \tilde{\gamma}^{-1} \tilde{\eta}_2^{-1})\\
  &= \tilde{\gamma} \tilde{\eta}_1 \tilde{\gamma}^{-1} (\tilde{\gamma} \tilde{\eta}_2 \tilde{\gamma}^{-1} \tilde{\eta}_2^{-1})\tilde{\eta}_1^{-1}\\
  &= \tilde{\gamma} (\tilde{\eta}_1\tilde{\eta}_2) \tilde{\gamma}^{-1}(\tilde{\eta}_1\tilde{\eta}_2)^{-1} 
  \\&= B_{IJ}({\gamma}, {\eta}_1{\eta}_2),
\end{aligned}
\end{equation*}
using that  $B_{IJ}$ takes values in the center of $\widetilde{LG}$.
  Hence, $B_{IJ}$ is a bihomomorphism.

Since $G$ is semisimple, \cref{TheoremSemisimpleNoHoms} implies that $B_{IJ}$ must be constant on the connected components of $L_IG$ and $L_JG$.
It follows that there exists a unique bihomomorphism 
\begin{equation*}
B^0_{IJ}: \pi_0(L_I G) \times \pi_0(L_J G) \longrightarrow \U(1) \qquad \text{with} \qquad B^0_{IJ}([\gamma], [\eta])=B_{IJ}(\gamma, \eta).
\end{equation*}

Since $I$ and $J$ are intervals, then the inclusion $L_IG \to \Omega G$ is a homotopy equivalence, hence induces an isomorphism $\pi_0(L_IG) \cong \pi_0(\Omega G) = \pi_1(G)$.
We conclude that in this case, there exists a unique bihomomorphism $b_{IJ}$ on $\pi_1(G)$ such that
\begin{equation*}
  B^0_{IJ}([\gamma], [\eta]) = b_{IJ}([\gamma], [\eta]).
\end{equation*}
whenever $\gamma \in L_I G$ and $\eta \in L_JG$.
Another way to say this is that, given an interval $I \subset S^1$, each element of $\pi_1(G)$ has a representative $\gamma$ supported in $L_IG$, and two such representatives of the same element of $\pi_1(G)$ are already homotopic in $L_IG$.
We now show that all these bihomomorphism $b_{IJ}$ is independent of the choice of $I$ and $J$.
\begin{enumerate}[(a)]
\item
\label{ZeroCase}
First observe that if $I^\prime \subseteq I$ and $J^\prime \subseteq J$, then $B_{I^\prime J^\prime}$ is the restriction of $B_{IJ}$ to $L_{I^\prime} G \times L_{J^\prime} G$.

\item 
\label{FirstCase}
Suppose now that $I, J$ and $I^\prime, J^\prime$ are two pairs of disjoint intervals of $S^1$ with the property that $I\cap I^\prime$ and $J \cap J^\prime$ are non-empty.
Then, using \eqref{ZeroCase}, we see that for $\gamma \in L_{I \cap I^\prime} G$, $\eta \in L_{J \cap J^\prime} G$, we have
\begin{equation*}
  b_{I^\prime J^\prime}([\gamma], [\eta]) = B_{I^\prime J^\prime}(\gamma, \eta) = B_{I \cap I^\prime, J \cap J^\prime}(\gamma, \eta) = B_{IJ}(\gamma, \eta) = b_{IJ}([\gamma], [\eta]).
\end{equation*}
Hence $b_{IJ} = b_{I^\prime J^\prime}$.
\item
\label{SecondCase}
Next we show that $b_{IJ} = b_{JI}$.
Choose $s \in I$ and $t \in J$ and let $K, L \subset S^1$ be the two disjoint intervals such that $\partial K = \partial L = \{s, t\}$.
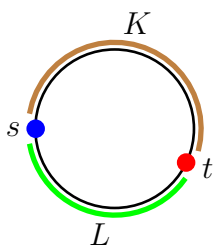
\begin{equation*}
\begin{tikzpicture}
 \draw[line width=1pt] (-2,0) circle (30pt);
 \filldraw[color=blue, line width=1pt] (-3.05,0) circle (3pt);
 \node at  (-3.35,0) {$s$};
 \filldraw[color=red, line width=1pt] (-1.05,-0.45) circle (3pt);
 \node at  (-0.76,-0.52) {$t$};
    \draw [line width=2pt,brown,domain=-15:170] plot ({1.15*cos(\x)-2}, {1.15*sin(\x)});
 \node at  (-1.7,1.4) {$K$};
   \draw [line width=2pt,green,domain=190:325] plot ({1.15*cos(\x)-2}, {1.15*sin(\x)});
 \node at  (-2.2,-1.4) {$L$};
\end{tikzpicture}
\end{equation*}
By construction, all the intersections $I \cap K$, $K \cap J$, $J \cap L$, $L \cap I$ are non-empty.
Therefore, by \eqref{FirstCase}, $b_{IJ} = b_{KL} = b_{JI}$. 
\item Now let $I, J$ and $I^\prime, J^\prime$ be two arbitrary pairs of disjoint intervals of $S^1$.
Choose pairwise distinct points $s \in I$, $t \in J$, $s^\prime \in I^\prime$, $t^\prime \in J^\prime$.
There are two different possible basic configurations for $s, s^\prime, t, t^\prime$, as depicted below.
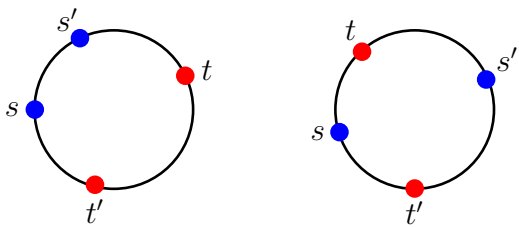
\begin{equation*}
\begin{tikzpicture}
 \draw[line width=1pt] (-2,0) circle (30pt);
 \filldraw[color=blue, line width=1pt] (-3.05,0) circle (3pt);
 \node at  (-3.35,0) {$s$};
 \filldraw[color=blue, line width=1pt] (-2.45,0.95) circle (3pt);
 \node at  (-2.62,1.2) {$s^\prime$};
 \filldraw[color=red, line width=1pt] (-1.05,0.45) circle (3pt);
 \node at  (-0.76,0.52) {$t$};
 \filldraw[color=red, line width=1pt] (-2.25,-1) circle (3pt);
 \node at  (-2.25,-1.35) {$t^\prime$};
 \draw[line width=1pt] (2,0) circle (30pt);
  \filldraw[color=blue, line width=1pt] (2.95,0.4) circle (3pt);
  \node at  (3.23,0.67) {$s^\prime$};
 \filldraw[color=blue, line width=1pt] (1,-0.3) circle (3pt);
 \node at  (0.71,-0.32) {$s$};
 \filldraw[color=red, line width=1pt] (1.3,0.77) circle (3pt);
 \node at  (1.15,1.05) {$t$};
 \filldraw[color=red, line width=1pt] (2.0,-1.05) circle (3pt);
 \node at  (2.0,-1.37) {$t^\prime$};
\end{tikzpicture}
\end{equation*}
In the first configuration, we can choose disjoint intervals $K, L \subset S^1$ such that $s, s^\prime \in K$ and $t, t^\prime \in L$.
Then, by construction, $K$ has non-empty intersection with both $I$ and $I^\prime$ and $L$ has non-empty intersection with both $J$ and $J^\prime$.
Consequently, by \eqref{FirstCase}, we obtain $b_{IJ} = b_{KL} = b_{I^\prime J^\prime}$.
In the second configuration, we can choose disjoint intervals $K, L \subset S^1$ such that $s, t^\prime \in K$ and $t, s^\prime \in L$.
Then $K$ has non-empty intersection with both $I$ and $J^\prime$ and $L$ has non-empty intersection with both $J$ and $I^\prime$.
Therefore, by \eqref{FirstCase} and \eqref{SecondCase}, $b_{IJ} = b_{KL} = b_{J^\prime I^\prime} = b_{I^\prime J^\prime}$.
\end{enumerate}
We conclude that the bihomomorphism $b_{IJ}$ is independent of the choice of disjoint intervals $I, J \subset S^1$. 
We write $b$ for this bihomomorphism on $\pi_1(G)$.

It follows from the definition of the bihomomorphisms $B_{IJ}$ that 
\begin{equation*}
  b_{IJ}(g, h) = b_{JI}(h, g)^{-1}, \qquad g, h \in \pi_1(G)
\end{equation*}
for any pair of disjoint intervals $I, J \subset S^1$.
%
%
Since $b = b_{IJ} = b_{JI}$, this shows that the bihomomorphism $b$ is skew.
\end{proof}

\begin{theorem}
\label{TheoremBihomomorphism}
A central extension  $\widetilde{LG}$  of the loop group $LG$ of a semisimple Lie group $G$ is disjoint commutative if and only if the bihomomorphism $b$ of \cref{LemmaObstructionBihomomorphism} vanishes. 
\end{theorem}

\begin{corollary}
\label{CorollarySemisimpleSimplyConnectedDisjointCommutative}
If $G$ is simply connected and semisimple, then all central extensions of $LG$ are disjoint commutative.
\end{corollary}

\begin{proof}[of \cref{TheoremBihomomorphism}]
It is clear that $b$ is trivial when $\widetilde{LG}$ is disjoint commutative.
If $b$ is trivial, then the bihomomorphisms $B_{IJ}^0$ (and consequently the $B_{IJ}$) for disjoint intervals $I,J$ are trivial as well, so that
 $\widetilde{LG}$ is disjoint commutative for intervals.

It remains to treat the case of general disjoint subsets $I,J \subset S^1$.
Observe that $L_I G = L_{I^\circ} G$,    where $I^\circ$ is the interior of $I$, hence we can assume throughout that $I,J \subseteq S^1$ are open.
Suppose now that $I = I_1 \sqcup I_2 \sqcup \dots$ is a disjoint union of possibly infinitely many intervals. 
Then $L_{I_1 \sqcup \dots \sqcup I_n}G \cong L_{I_1} G \times \cdots \times L_{I_n} G$ for each $n \in \N$.
Moreover, the union
\begin{equation*}
\bigcup_{n=1}^\infty L_{I_1 \sqcup \dots \sqcup I_n}G \subset L_I G
 \end{equation*}
 is dense.
  This implies that the group of connected components of $L_IG$ is the direct \emph{sum}
 \begin{equation*}
 \pi_0(L_I G) = \bigoplus_{k=1}^\infty \pi_0(L_{I_k} G).
 \end{equation*}
 We therefore obtain that if $J= J_1 \sqcup J_2 \sqcup \dots \subset S^1$ is another such subset, then the bihomomorphism $B^0_{IJ}$ is determined by the bihomomorphisms $B^0_{I_kJ_l}$, $j, k \in \N$, which vanish by assumption.
\end{proof}

Let $\widetilde{LG}$ be a central extension of $LG$.
Any $\U(1)$-valued group 2-cocycle $\kappa$ on $\pi_1(G)$ can be used to modify the group product of $\widetilde{LG}$ according to the formula
\begin{equation} 
\label{ModifiedProduct}
  \Phi \star \Psi = \kappa([\pi(\Phi)], [\pi(\Psi)]) \cdot \Phi \Psi.
\end{equation}
We assume throughout that $\kappa$ is normalized, in the sense that $\kappa(g, e) = \kappa(e, g) = \kappa(e, e) = 1$.
This is equivalent to requiring that the unit elements for the two products coincide.
%
%
Normalization is no serious restriction as every cocycle is cohomologous to a normalized one.
Moreover, if $\kappa = d\rho$ for some $\U(1)$-valued 1-cocycle, then the central extension $\widetilde{LG}$ with the modified product \eqref{ModifiedProduct} is isomorphic to the original central extension.

\begin{lemma}
\label{LemmaModifiedProduct}
Let $G$ be a semisimple Lie group and let $b$ be the obstruction bihomomorphism of \cref{LemmaObstructionBihomomorphism} for a central extension $\widetilde{LG}$.
Then, the obstruction bihomomorphism $b^\prime$ for the central extension with the modified product \cref{ModifiedProduct} is given by
\begin{equation*}
  b^\prime(g, h) = b(g, h) \cdot \sskew \kappa (g, h)^{-1},
\end{equation*}
where
\begin{equation*}
\sskew \kappa (g, h) := \kappa(g, h) \kappa(h, g)^{-1}
\end{equation*}
is the \emph{skew} of $\kappa$.
\end{lemma}

It is well-known that the skew of a 2-cocycle on an \emph{abelian} group is always a bihomomorphism; 
notice here that $\pi_1(G)$ is abelian as $G$ is a Lie group.

\begin{proof}
Let $\Phi \in L_{I} G$ and $\Psi \in L_{J} G$ for $I=(0,\pi)$ and $J=(\pi,2\pi)$, and let $g = [\pi(\Phi)], h = [\pi(\Psi)] \in \pi_1(G)$.
The inverses of $\Phi$ and $\Psi$ with respect to the modified product \cref{ModifiedProduct} are
\begin{equation*}
  \Phi^{\star -1} = \kappa(g, g^{-1})^{-1} \Phi^{-1}, \qquad \Psi^{\star-1} = \kappa(h, h^{-1}) \Psi^{-1}.
\end{equation*}
Then, using that $\pi_1(G)$ is abelian, 
\begin{equation*}
\begin{aligned}
  b^\prime(g, h) &= B^\prime_{IJ}(\pi(\Phi), \pi(\Psi)) 
  \\&= \Phi\star \Psi \star \Phi^{\star -1}  \star \Psi^{\star -1}\\
  &= \kappa(g, g^{-1})^{-1} \kappa(h, h^{-1})^{-1}\Phi \star \Psi \star \Phi^{-1} \star \Psi^{-1}\\
  &= \kappa(g, g^{-1})^{-1} \kappa(h, h^{-1})^{-1}\kappa(g, h) \kappa(g h, g^{-1}) \underbrace{\kappa(g h g^{-1}, h^{-1})}_{= \kappa(h, h^{-1})} \Phi \Psi\Phi^{-1}\Psi^{-1}\\
  &= \kappa(g, g^{-1})^{-1}\kappa(g, h) \kappa(g h, g^{-1})  b(g, h)
  \end{aligned}
\end{equation*}
Since $\kappa$ is a group cocycle, we have
\begin{equation*}
  \kappa(g h, g^{-1}) = \kappa(hg, g^{-1}) = \kappa(h, g)^{-1}\kappa(h, 1)\kappa(g, g^{-1}) = \kappa(h, g)^{-1}\kappa(g, g^{-1}),
\end{equation*}
as $\kappa$ is assumed to be normalized.
Plugging this into the previous formula yields the desired result.
\end{proof}

\begin{example}
The above results provide many examples of central extensions of $LG$ for non-simply connected Lie groups $G$ that are not disjoint commutative.
For example, suppose we have $\xi \in \U(1)$ and $p,q\in \Z$ such that $\xi^p = \xi^q = 1$.
Then, the group 2-cocycle $\kappa$ on $\Z/p\Z \times \Z/q\Z$ given by
\begin{equation*}
  \kappa((k_1, k_2), (l_1, l_2)) = \xi^{k_1l_2},
\end{equation*}
has non-trivial skew.
For example, with the choices $\xi = -1$ and $p = q = 2$, the trivial central extension of $L(\SO(m)\times \SO(n))$ ($m, n \geq 3$) modified by the cocycle $\kappa$ provides a central extension that is not disjoint commutative.
\end{example}

\begin{example}
\label{ExampleNotSemisimple}
Things change completely upon leaving the realm of semisimple Lie groups.
An example of a non-disjoint commutative central extension in the case that $G$ has trivial fundamental group is the following. 
Consider $G= \R^+$ and let $\widetilde{LG}$ be the central extension corresponding to the group cocycle
\begin{equation*}
  \kappa(\gamma, \eta) = \exp(i \log \gamma(s) \cdot \log \eta(t)),
\end{equation*}
for $s, t \in S^1$ fixed.
Since $L\R^+$ is abelian, the bihomomorphism $B_{IJ}$ from \cref{BihomomorphismB} is the restriction of a bihomomorphism $B$ defined on all of $LG$, which is just the skew of $\kappa$.
This is non-zero whenever $s \neq t$.
\end{example}

\begin{example}
A further example of a central extension of $L\U(1)$ that is not disjoint commutative is given as Example~4.12 in \cite{WaldorfTransgressive}.
\end{example}

Recall that a bihomomorphism $b$ on an abelian group $K$ is \emph{alternating} if $b(g, g) = 1$.
Any alternating bihomomorphism is skew, but the converse is not always true in the presence of 2-torsion in the target.
By definition, the skew of a group cocycle is always alternating.
Moreover, an easy calculation shows that the skew of a coboundary is zero.
Hence we obtain a well-defined group homomorphism
\begin{equation}
\label{SurjectiveToAlt}
  H^2(K, \U(1)) \longrightarrow \Alt^2(K, \U(1))
\end{equation}
from the second group cohomology of $K$ to the group of alternating bihomomorphisms on $K$.
It is a fact that this group homomorphism is always surjective \cite[Proposition 3.3]{NeebQuantumTori}.

Given a non-disjoint commutative central extension $\widetilde{LG}$, one may ask whether we can modify the product by a group cocycle $\kappa$ such that $\widetilde{LG}$ becomes disjoint commutative.
To investigate this question, consider the map that assigns to a central extension the obstruction bihomomorphism from Lemma~\ref{LemmaObstructionBihomomorphism}.
By Lemma~\ref{LemmaModifiedProduct} and the surjectivity of \eqref{SurjectiveToAlt}, this map descends to a group homomorphism
\begin{equation}
\label{StronglyDisjointCommutative}
  \frac{h_0\bigl(\CExt(LG)\bigr)}{H^2(\pi_1(G), \U(1))} ~~\longrightarrow~~ \frac{\Skew^2(\pi_1(G), \U(1))}{\Alt^2(\pi_1(G), \U(1))},
\end{equation}
where $\Skew^2(\pi_1(G), \U(1))$ denotes the group of skew bihomomorphisms on $\pi_1(G)$ and $H^2(\pi_1(G), \U(1))$ acts on the set of isomorphism classes of central extensions by modifying the product according to \eqref{ModifiedProduct}.
It is easy to see that the quotient on the right hand is just the 2-torsion subgroup of $\pi_1(G)$.
%
Thus, we obtain the following result.

\begin{theorem}
\label{ThmModifiedProduct}
Let $G$ be a semisimple Lie group.
If $\pi_1(G)$ has no 2-torsion, then for any central extension $\widetilde{LG}$ of $LG$, there exists a group 2-cocycle $\kappa \in H^2(\pi_1(G), \U(1))$ such that $\widetilde{LG}$ with the product modified by $\kappa$ is disjoint commutative.
\end{theorem}

\begin{example}
\label{ExampleSOd}
Consider $G = \SO(d)$ for $d \geq 5$. 
Then the group of isomorphism classes of central extensions of $L\SO(d)$ is isomorphic to $\Z \times \Z_2$, where the first factor is called the \emph{level} and the second factor  comes from central extensions of $\SO(d)$ (compare Lemma~4.8 of \cite{ludewig2022clifford}).
Since $\pi_1(\SO(d)) = \Z_2$ and $H^2(\Z_2, \U(1)) = 0$, there are no non-trivial product modifications by group cocycles.
It turns out that the obstruction bihomomorphism of the generator of $h_0(\CExt(L\SO(d)))$ is the (non-alternating) $\U(1)$-valued skew bihomomorphism
\begin{equation*}
  b(k_1, k_2) = (-1)^{k_1k_2}
\end{equation*}
on $\Z_2$.
Hence, a central extension of $L\SO(d)$ is disjoint commutative if and only if it is of even level.
Of course, by Corollary~\ref{CorollarySemisimpleSimplyConnectedDisjointCommutative}, all central extensions of $L\SO(d)$ become disjoint commutative when pulled back along $L \Spin(d) \to L\SO(d)$.
\end{example}

Example~\ref{ExampleSOd} shows that the group homomorphism \eqref{StronglyDisjointCommutative} is  surjective for $G=\SO(d)$.
We do not know whether it is surjective for every semisimple Lie group $G$.


%
%


\section{Lie 2-groups from loop group extensions}

\label{section-3}

This section contains the main result of the present article, namely, the construction of Lie 2-groups from loop group extensions. In \cref{SectionStrict2GroupsCrossedModules} we recall the relevant facts about crossed modules and Lie 2-groups, and \cref{SectionCrossedModulesFromCExt} contains the main construction. \Cref{SectionFusionFactorizations} concerns the notion of a fusion factorization that allows one to give our Lie 2-groups a more convenient form. In \cref{classification-of-Lie-2-groups} we show that our Lie 2-groups deliver 3-connected covering groups, in particular, models for the string 2-group. 

\subsection{Strict Lie 2-groups and crossed modules}
\label{SectionStrict2GroupsCrossedModules}

We recall that a \emph{strict Lie 2-group} is a groupoid $\Gamma = (\Gamma_0, \Gamma_1, s, t, i, \circ, \inv)$ whose set $\Gamma_0$ of objects and whose set $\Gamma_1$ of morphisms are (possibly Fr\'echet) Lie groups, whose source and target map $s, t : \Gamma_1 \to \Gamma_0$,  composition $\circ : \Gamma_1 \times_{t,s} \Gamma_1 \to \Gamma_1$, identity map $i : \Gamma_0 \to \Gamma_1$, and inversion (with respect to composition) $\inv : \Gamma_1 \to \Gamma_1$ are all smooth group homomorphisms.
We note  that if $\Gamma_1$ and $\Gamma_0$ are finite-dimensional, then the fibre product $\Gamma_1 \times_{t,s} \Gamma_1$ exists since $s$ and $t$ are surjective Lie group homomorphisms, hence submersions; in the infinite-dimensional setting, the existence of the fibre product is a further assumption that we need to put. We also note that the group
\begin{equation*}
  \pi_1(\Gamma) := \ker(s) \cap \ker(t) \subseteq \Gamma_1
\end{equation*}
is abelian.

 When constructing  strict Lie 2-groups  it is worthwhile to notice that composition and inversion are already determined by the remaining structure. Indeed, it is straightforward to see that 
\begin{equation} \label{FormulaForComposition}
x \circ y = x \, i(s(x))^{-1} y = x \, i(t(y))^{-1} y \text{,}
\end{equation}
for composable morphisms $x, y \in \Gamma_1$, i.e., morphisms such that $s(x) = t(y)$.
It follows from this that the inverse of a morphism $x \in \Gamma_1$ with respect to composition satisfies
\begin{equation}
\label{FormulaForInversion}
        \inv(x)=i(s(x))x^{-1}i(t(x))\text{.}
\end{equation}
Moreover, in a strict Lie 2-group the subgroups $\ker(s)$ and $\ker(t)$ of $\Gamma_1$ commute: let $x \in \ker(s)$, $y \in \ker(t)$, and let $e\in \Gamma_0$ be the unit element. Then 
\begin{equation} \label{eq:KerSKerT}
        y x  = (e \circ y) (x \circ e) = (e \cdot x) \circ (y \cdot e) = x \circ y = x \, i(s(x))^{-1} y = xy.
\end{equation}
We have the following converse of these three observations.

\begin{lemma}
        \label{LemmaMinimalData2Group}
        Suppose $\Gamma_0$ and $\Gamma_1$ are Lie groups and $s, t:\Gamma_1 \to \Gamma_0$ and $i: \Gamma_0\to\Gamma_1$ are smooth group homomorphisms such that:
        \begin{enumerate}[{\normalfont (a)}]

                \item 
                        \label{cor:MinimalData2Group:a}
                        $s \circ i = \id_{\Gamma_0} = t \circ i$.

                \item
                        \label{cor:MinimalData2Group:b}
                        $\mathrm{ker}(s)$ and $\mathrm{ker}(t)$ are commuting Lie subgroups.

        \end{enumerate} 
        Then, together with the composition defined by \cref{FormulaForComposition} and the inversion defined in \cref{FormulaForInversion}, this structure constitutes a strict Lie 2-group. 
\end{lemma}

\begin{proof}
First of all we prove that the fibre product $\Gamma_1 \times_{t,s} \Gamma_1$ exists in the category of Fr\'echet Lie groups.  
We consider $U := \mathrm{ker}(s) \times \mathrm{ker}(s) \times \Gamma_0$ equipped with the maps $f, g:U \to \Gamma_1$ defined by $f(x, y, z) :=x i(z) $ and $g(x, y, z) := yi(t(x)z)$. Then we have
\begin{equation*}
(s \circ g)(x, y, z) =s(yi(t(x)z))=t(x)z=t(xi(z))=(t \circ f)(x, y, z)\text{.}
\end{equation*}
By \eqref{cor:MinimalData2Group:b} we see that  $U$ is a Fr\'echet manifold, and  the maps $f$ and $g$ are clearly smooth. 
Moreover, we turn $U$ into a Fr\'echet Lie group,  and $f$ and $g$ into group homomorphisms, by declaring
\begin{equation*}
(x_1, y_1, z_1) \cdot (x_2, y_2, z_2) := (x_1 i(z_1)x_2i(z_1)^{-1}, y_1x_1i(z_1)y_2i(z_1)^{-1}x_1^{-1}, z_1z_2)\text{.}
\end{equation*}
Now we assume that
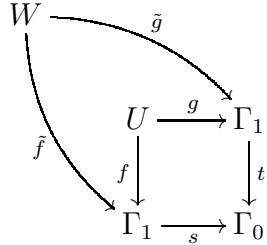
\begin{equation*}
\xymatrix{W \ar@/_1pc/[ddr]_{\tilde f} \ar@/^1pc/[rrd]^{\tilde g} \\ & U \ar[d]_{f} \ar[r]^{g} & \Gamma_1 \ar[d]^{t} \\ & \Gamma_1 \ar[r]_{s} & \Gamma_0 }
\end{equation*}
is a commutative diagram in the category of Fr\'echet Lie groups. 
We define
\begin{equation*}
h: W \to U;\quad h(w) := (w_1i(s(w_1))^{-1}, w_2i(s(w_2))^{-1}, s(w_1)) 
\end{equation*}
where $w_1 := \tilde f(w)$ and $w_2 := \tilde g(w)$. This is a smooth group homomorphism such that $g \circ h = \tilde g$ and $f \circ h = \tilde f$. 
It is straightforward to check that this map is unique with this property. This shows that $U$ is the required fibre product; in particular, it exists. It is then  easy to see that   the composition defined by \cref{FormulaForComposition} is smooth and that -- using \eqref{cor:MinimalData2Group:a} -- turns $\Gamma$ into a Fr\'echet Lie
groupoid.

        The commutativity in condition \eqref{cor:MinimalData2Group:b} is used in order to show that composition is a group homomorphism:
        Let  $x_1,  x_2,  y_1,  y_2 \in \Gamma_1$ with $s(x_1) = t(y_1)$, $s(x_2) = t(y_2)$.
        Observe that $x_2 \, i(s(x_2)) \in \ker(s)$ and $i(t(y_1))^{-1} y_1 \in \ker(t)$.
        Therefore, we can calculate
        \begin{align*}
                x_1 x_2 \circ y_1 y_2 &= x_1 x_2 \, i(t(y_1 y_2))^{-1} y_1 y_2 \\
                                      &= x_1 x_2 \, i(t( y_2))^{-1} i(t(y_1))^{-1}  y_1 y_2\\
                                      &= x_1 x_2 \, i(s( x_2))^{-1} i(t(y_1))^{-1}  y_1 y_2\\
                                      &= x_1  i(t(y_1))^{-1}  y_1 x_2 \, i(s( x_2))^{-1} y_2\\
                                      &= (x_1 \circ y_1) \cdot (x_2 \circ y_2),
        \end{align*}
        where in the second last step, we used (b). 
%
%
%
\end{proof}

%
%
%

        Another way to present (Fr\'echet) Lie 2-groups is in terms of   crossed modules of (Fr\'echet) Lie groups. 
Recall that a \emph{crossed module} $X$  of Fr\'echet Lie groups consists of a pair of Fr\'echet Lie groups $G$ and $H$ together with a smooth group homomorphism $t : H \to G$ and \emph{crossed module action} of $G$ on $H$, i.e., a smooth map $\alpha : G \times H \to H$, such that $\alpha$ is an action of $G$ on $H$ by group homomorphisms, and
\begin{equation} \label{CrossedModuleAxioms}
t(\alpha_g(h))=gt(h)g^{-1} \qquad \text{and} \qquad
\alpha_{t(h)}(k)=hkh^{-1}
\end{equation}
hold for all $g\in G$ and $h, k\in H$, where $\alpha_g(h):=\alpha(g, h)$.  
The first property means that $t$ is $G$-equivariant for the $G$-action $\alpha$ on $H$ and the conjugation action of $G$ on itself.
The second property is called the \emph{Peiffer identity}.

Observe that for a crossed module $X$, the Peiffer identity implies that  $A:=\ker(t)$ lies in the center of $H$ and, in particular, is abelian.
By $G$-equivariance of $t$, the $G$-action $\alpha$ restricts to an action on $A$.
The crossed module is called \emph{central} if this action of $G$ on $A$ is trivial.

There is an adjoint equivalence
\begin{equation} 
\label{AdjunctionLieCross}
\xymatrix{
  \mathcal{X} : \Lietwogroups \ar@<3pt>[r] & \ar@<3pt>[l] \XMod : \mathcal{G}
}
\end{equation}
between the category $\Lietwogroups$ of Fr\'echet Lie 2-groups and the category $\XMod$ of  crossed modules of Fr\'echet Lie groups, when both are equipped with the obvious notion of strict morphisms. 
For plain crossed modules of sets, this is the Brown-Spencer theorem \cite{BrownSpencerCrossed}, which has been generalized to crossed modules ambient to another category  by Janelidze \cite{Janelidze2003}; here we use it in the Fr\'echet Lie group setting.
Explicitly, the equivalence \cref{AdjunctionLieCross} is given by:
\begin{equation*}
\mathcal{G}(H, G, t, \alpha) :=\left(
\quad
\begin{aligned}
   \Gamma_1 := &\; H \rtimes_\alpha G \\
   \Gamma_0 := &\; G
   \\\strut
\end{aligned}
\quad
\begin{aligned}
   s(h, g) :=& g~~~~~ \\
    t(h, g) :=& t(h) g \\ 
    i(g) :=& (e, g)
\end{aligned}
\quad
\right)
\end{equation*}
\begin{equation*}
\mathcal{X}(\Gamma, s, t, i) :=
\left(\quad
\begin{aligned}
 H &:=  \ker(s)\subseteq \Gamma_1 \\ G &:= \Gamma_0\\
 t &:= (t: \Gamma_1 \to \Gamma_0)_{|\ker(s)}\\
 \alpha_g(h) &:= i(g) h\, i(g)^{-1}
\end{aligned}
\quad
 \right)
\end{equation*}
The above description of $\mathcal{G}$ uses \cref{LemmaMinimalData2Group}, which applies here since the Lie subgroups $\mathrm{ker}(s)=\{(h, 1) \,|\, h\in H\}\cong H$ and $\mathrm{ker}(t)=\{(h^{-1},t(h))\,|\,h\in H\}\cong H$ commute. 
It is worthwhile to look at the unit and counit maps
\begin{equation*}
\epsilon : \mathcal{X}\mathcal{G} \Rightarrow \id_{\XMod} \qquad \eta : \id_{\Lietwogroups} \Rightarrow \mathcal{G} \mathcal{X}
\end{equation*}
of the adjunction \cref{AdjunctionLieCross}.
While the formula for the unit $\epsilon$ is obvious, the counit $\eta$ is given at a Lie 2-group $\Gamma$ by the strict Lie 2-group isomorphism
\begin{equation*}
  \eta_\Gamma = \left( \begin{aligned} \Gamma_1 &\rightarrow \ker(s) \rtimes_\alpha \Gamma_0, & h &\mapsto (h \, i(s(h))^{-1}, s(h)) \\ \Gamma_0 &\to \Gamma_0, &  g &\mapsto g\end{aligned}\right).
\end{equation*}

\begin{example}
\label{ex:2groupsAbelian}
Given any abelian  Lie group $A$, setting $\Gamma_0 = \{e\}$, $\Gamma_1 = A$ (and trivial $s, t, i$) give a strict  Lie 2-group denoted by $BA$.
The corresponding crossed module is $A \to \{e\}$, with the (necessarily trivial) action.
Observe that $A$ is forced to be abelian by the requirement of \cref{LemmaMinimalData2Group} \eqref{cor:MinimalData2Group:b}.
\end{example}

\begin{example}
\label{ex:2-groupsOrdinaryGroups}
Any  Lie group $G$ can be viewed as a strict  Lie 2-group, denoted $G\dis$, by setting $\Gamma_0 = \Gamma_1 = G$ and $s=t=i=\id_G$.
The corresponding crossed module is $\{e\} \to G$.
%
%
\end{example}

\subsection{Crossed modules from  loop group extensions}

\label{SectionCrossedModulesFromCExt}

Let again $G$ be a connected (finite-dimensional) Lie group and let 
\begin{equation*}
1 \to \U(1) \to \widetilde{\Omega G} \stackrel\pi\to \Omega G \to 1
\end{equation*}
be a Fr\'echet central extension of the based loop group $\Omega G$.
We will now describe how to use this central extension to produce crossed module of Fr\'echet Lie groups.
For a Lie subgroup $H \subset \Omega G$, we write 
\begin{equation*}
\widetilde{H} := H \times_{\Omega G} \widetilde{\Omega G}
\end{equation*}
for the pullback of $\widetilde{\Omega G}$ to $H$, and address an element $(h,\Phi)\in \widetilde H$ by just $\Phi$.

We identify $P_eG^{[2]}$ with a subgroup of $\Omega G$ using the injective map $\cup: P_eG^{[2]} \to \Omega G$, and hence consider, in the above notation, the pullback 
\begin{equation*}
\widetilde{P_eG^{[2]}} = P_eG^{[2]} \times_{\Omega G} \widetilde{\Omega G}\text{.}    
\end{equation*}
 To begin with, we have canonical maps
\begin{equation}
\label{SourceTargetMaps}
        \xymatrix{ 
        \widetilde{P_e G^{[2]}} \ar@<3pt>[r]^{s} \ar@<-3pt>[r]_{t} & P_e G
        }
        , \qquad 
        \begin{aligned} s(\Phi) &= \gamma_2, \\ t(\Phi) &= \gamma_1, \end{aligned}
        \qquad \text{whenever} \quad \pi(\Phi) = \gamma_1 \cup \gamma_2.
\end{equation}
We note  that 
\begin{equation}
\label{kernel-of-s}
\ker(s) = \widetilde{\Omega_{(0, \pi)}G}.
\end{equation}

As for any central extension, the conjugation action of $\widetilde{\Omega G}$ on itself descends to a smooth action of $\Omega G$.
This action is trivial on $\U(1) \subset \widetilde{\Omega G}$ and restricts to the subgroups $\widetilde{\Omega_I G}$, for any subset $I \subseteq S^1$.
Pulling back along the \quot{diagonal} group homomorphism 
\begin{equation*}
P_e G \to  P_eG^{[2]} \stackrel\cup\to \Omega G\text{,}
\end{equation*}
where $\cup$ is defined in \eqref{CupMap}, we obtain an action of $P_e G$ on $\widetilde{\Omega G}$. 
The restriction of this action to $\widetilde{\Omega_{(0, \pi)} G}$ will be denoted by $\alpha$, and will be called the \emph{canonical action} associated to $\widetilde{\Omega G}$.
Explicitly, it is given by
\begin{equation} \label{CanonicalAction}
  \alpha: P_eG \times \widetilde{\Omega_{(0, \pi)} G} \longrightarrow \widetilde{\Omega_{(0, \pi)} G}, \qquad \alpha_\gamma(\Phi)  = \widetilde{\gamma\cup\gamma} \cdot \Phi \cdot (\widetilde{\gamma\cup\gamma})^{-1},
\end{equation}
where $\widetilde{\gamma\cup\gamma}$ is any lift of $\gamma \cup \gamma$ to $\widetilde{\Omega G}$.
As the choice of lift is unique up to an element in the center of $\widetilde{\Omega G}$, the right hand side of \cref{CanonicalAction} is independent of the choice of lift.

\begin{remark}
We emphasize that the construction of the canonical action $\alpha$ is \emph{much} simpler than the construction in \cite[Lemma~24]{BCSS} and, in particular, that it does not depend on any additional data or a particular model for the central extension (compare also \cite[Prop.~4.3.2]{PressleySegal}).
That the canonical action  coincides with the action from  \cite[Lemma~24]{BCSS} will be discussed in detail in \cref{Comparison-with-BCSS}.
\end{remark}

The map $t$ intertwines the canonical action $\alpha$ with the conjugation action of $P_e G$ on itself,
\begin{equation*}
t (\alpha_\gamma(\Phi)) = \gamma \cdot t(\Phi) \cdot \gamma^{-1}.
\end{equation*}
However, the canonical action $\alpha$ does \emph{not} generally satisfy the Peiffer identity 
\begin{equation} \label{PeifferIdentity}
 \alpha_{t(\Psi)}(\Phi) = \Psi \cdot \Phi \cdot \Psi^{-1}.
\end{equation}
Instead, we have the following lemma.

\begin{lemma}
  If the central extension $\widetilde{\Omega G}$ is disjoint commutative, then the canonical action $\alpha$ of \cref{CanonicalAction} satisfies the Peiffer identity.
\end{lemma}

\begin{proof}
Let $\Psi, \Phi \in \ker(s) = \widetilde{\Omega_{(0, \pi)} G}$ and write $\gamma = t(\Psi)$. 
Then
\begin{equation*}
\begin{aligned}
  \alpha_{t(\Psi)}(\Phi) &= \widetilde{\gamma \cup \gamma} \cdot \Phi \cdot \widetilde{\gamma \cup \gamma}^{-1}\\
  &= (\widetilde{\gamma \cup \gamma} \cdot \Psi^{-1}) \cdot (\Psi \Phi \Psi^{-1}) \cdot (\Psi \cdot \widetilde{\gamma \cup \gamma}^{-1}).
\end{aligned}
\end{equation*}
The middle term is contained in $\ker(s) = \widetilde{\Omega_{(0, \pi)} G}$, while the outer terms are contained in $\ker(t) = \widetilde{\Omega_{(\pi, 2\pi)} G}$.
Hence, by disjoint commutativity, these terms commute, leading to the desired result.
\end{proof}

Finally, we observe that the canonical action $\alpha$ is trivial on the central subgroup $\U(1)\subset \widetilde{\Omega_{(0, \pi)} G}$.
Thus, we obtain the following result.
\begin{theorem}
\label{CanonicalCrossedModule} 
If $\widetilde{\Omega G}$ is a disjoint commutative central extension of $\Omega G$, then the Lie group homomorphism $t: \widetilde{\Omega_{(0, \pi)} G} \to P_e G$ and the canonical action $\alpha$ of \cref{CanonicalAction} form a central crossed module of Fr\'echet Lie groups, denoted by $X(\widetilde{\Omega G})$. 
\end{theorem}

Next we study the question if there are other options for the crossed module action $\alpha$.

\begin{theorem}
\label{ThmUniqueAction}
Let $G$ be a semisimple Lie group and let $\widetilde{\Omega G}$ is a disjoint commutative central extension of $\Omega G$.
 Let moreover $\alpha^\prime$ be an  action of $P_e G$ on $\widetilde{\Omega_{(0, \pi)} G}$ turning 
\begin{equation*} 
t: \widetilde{\Omega_{(0, \pi)} G} \to P_e G
\end{equation*}
into a central crossed module. 
Then $\alpha^\prime$ coincides with the canonical action $\alpha$ of \cref{CanonicalAction}.
\end{theorem}

\begin{proof}
For $\gamma \in P_e G$ and $\eta \in \Omega_{(0, \pi)} G$, we define a map $\kappa_{\gamma}: \Omega_{(0, \pi)} G \to \widetilde{\Omega_{(0, \pi)} G}$ by 
  \begin{equation} \label{OmegaDefinition}
   \kappa_\gamma(\eta) := \alpha^\prime_{\gamma}(\tilde{\eta})\alpha_{\gamma}(\tilde{\eta})^{-1},
\end{equation}
where $\tilde{\eta}$ is any lift of $\eta$.
This is well-defined, as any two lifts of $\eta$ differ only by an element $z \in \U(1)$ and both actions are central, so $\alpha^\prime_\gamma(z) = z = \alpha_\gamma(z)$.
It is moreover smooth as $\widetilde{\Omega G}$ possesses smooth local sections.
As both $\alpha^\prime$ and $\alpha$ intertwine $t$ with the conjugation action of $P_eG$ on $\Omega_{(0,\pi)}G$, we have $t(\kappa_{\gamma}(\eta)) = \const_e$ for all $\gamma \in P_e G$, $\eta \in \Omega_{(0, \pi)} G$, hence $\kappa_\gamma$ takes values in $\U(1)$. Moreover,
$\kappa_{\gamma}$ is a group homomorphism: 
\begin{equation*}
\begin{aligned}
  \kappa_{\gamma}(t(\Phi)t(\Psi)) &= \alpha^\prime_\gamma(\Phi \Psi)\alpha_\gamma(\Phi \Psi)^{-1}\\
  &= \alpha^\prime_\gamma(\Phi)\underbrace{\alpha^\prime_\gamma( \Psi)\alpha_\gamma(\Psi )^{-1}}_{\in \U(1)}\alpha_\gamma(\Phi)^{-1}\\
  &= \alpha^\prime_\gamma(\Phi)\alpha_\gamma(\Phi)^{-1} \alpha^\prime_\gamma( \Psi)\alpha_\gamma(\Psi)^{-1}\\
  &= \kappa_{\gamma}(t(\Phi))\kappa_{\gamma}(t(\Psi)).
\end{aligned}
\end{equation*}
By \cref{TheoremSemisimpleNoHoms}, $\kappa_\gamma: \Omega_{(0,\pi)}G \to \U(1)$ must be the trivial group homomorphism for each $\gamma \in P_e G$.
Hence  $\alpha^\prime$ coincides with $\alpha$.
\end{proof}

Let $\XExt G$  be the subcategory of $\XMod$ consisting of those central crossed modules $(\widetilde{\Omega_{(0, \pi)} G}, P_e G, t, \alpha)$ in which $\widetilde{\Omega_{(0, \pi)} G}$ is a disjoint commutative central extension of $\Omega_{(0, \pi)} G$, and $t:\widetilde{\Omega_{(0, \pi)} G} \to P_eG$ is given as before; i.e., if $\Phi \in \widetilde{\Omega_{(0, \pi)} G}$ projects to $\gamma \cup \const_e$, then $t(\Phi)=\gamma$.
The morphisms are crossed module morphisms whose  map $P_eG \to P_eG$ is the identity, and whose map $\widetilde{\Omega_{(0, \pi)} G} \to \widetilde{\Omega_{(0, \pi)} G}'$ is a morphism of central extensions of $\Omega_{(0, \pi)} G$.  On the other side, we let $\DCCExt(\Omega G)$  denote the full subcategory of $\CExt(\Omega G)$ over all disjoint commutative central extensions of $\Omega G$.
\cref{CanonicalCrossedModule} establishes a functor
\begin{equation} \label{CrossedModuleAssignment}
 X : \DCCExt(\Omega G) \longrightarrow \XExt G.
\end{equation}
In order to see this, it suffices to observe  that any automorphism of a central extension $\widetilde{\Omega G}$ provides an automorphism of the restricted central extension $\widetilde{\Omega_{(0, \pi)} G}$ that intertwines the action $\alpha$.
%

\begin{corollary} 
If $G$ is simply connected and semisimple, the functor $X$ is an equivalence of categories,  $\DCCExt(\Omega G) \cong \XExt G$.
\end{corollary}

\begin{proof}
By \cref{LemmaAutoCExt}, the assumptions on $G$ imply that both $\DCCExt(\Omega G)$ and $\XExt G$ are groupoids with trivial automorphism groups. 
Therefore, we only have to show that the functor $X$ is a bijection on isomorphism classes of objects. 

If two crossed modules $X(\widetilde{\Omega G})$ and $X(\widetilde{\Omega G}^\prime)$ are isomorphic via an isomorphism in $\XExt G$, then this in particular implies that the restricted central extensions $\widetilde{\Omega_{(0, \pi)} G}$ and $\widetilde{\Omega_{(0, \pi)} G}^\prime$ are isomorphic.
But, by  \cref{LemmaRestrictionIso}, this implies that $\widetilde{\Omega G}$ and $\widetilde{\Omega G}$ are themselves isomorphic.
Hence the functor $X$ is injective.

Conversely, by the same \cref{LemmaRestrictionIso}, any central extension $\widetilde{\Omega_{(0, \pi)} G}$ of $\Omega_{(0, \pi)} G$ is the restriction of a central extension $\widetilde{\Omega G}$ of $\Omega G$. 
From the  proof of that lemma it is clear  that $\widetilde{\Omega G}$ is disjoint commutative if $\widetilde{\Omega_{(0, \pi)} G}$ is.
%
%
\end{proof}

\begin{remark}
The group homomorphism $\kappa_{\gamma}$ from the proof of \cref{ThmUniqueAction}, defined in \cref{OmegaDefinition}, can be defined for any two central crossed module actions $\alpha$ and $\alpha^\prime$ for the homomorphism $t:\widetilde{\Omega_{(0, \pi)} G} \to P_eG$, for any central extension $\widetilde{\Omega G}$ and without assuming that $G$ is semisimple.
As both $\alpha$ and $\alpha^\prime$ satisfy the Peiffer identity,  $\kappa_\gamma$ depends on $\gamma$ only through the endpoint $g = \gamma(\pi)$.
%
%
Varying $g$, we obtain a map
\begin{equation*}
\kappa : G \to \Hom(\Omega_{(0, \pi)} G, \U(1)).
\end{equation*}
The group $\Hom(\Omega_{(0, \pi)} G, \U(1))$ carries a right action of $P_e G$ given by pre-composition with the conjugation action on $\Omega_{(0, \pi)} G$, which descends to an action of $G$ as $\Omega_{(0, \pi)} G$ acts trivially.
One can then show that $\kappa$ is a diffeological group 1-cocycle with values in the right $G$-module $\Hom(\Omega_{(0, \pi)} G, \U(1))$, equipped with the functional diffeology.
%
%

Conversely, modifying $\alpha$ by a general $\Hom(\Omega_{(0, \pi)} G, \U(1))$-valued diffeological group cocycle $\kappa$ on $G$ according to formula \cref{OmegaDefinition} gives another crossed module action of $P_e G$ on $\Omega_{(0, \pi)} G$, and the resulting crossed module is isomorphic in $\XExt G$ to the previous one if and only if $\kappa$ is a coboundary.
\end{remark}

\subsection{Fusion factorizations}

\label{SectionFusionFactorizations}

Let $\widetilde{\Omega G}$ be a disjoint commutative central extension of $\Omega G$. 
In \cref{CanonicalCrossedModule} we have constructed a canonical crossed module $X(\widetilde{\Omega G})$   associated to $\widetilde{\Omega G}$.
The functor $\mathcal{G}$ from the adjunction \cref{AdjunctionLieCross} turns it into a  strict Lie 2-group. Explicitly, this Lie 2-group, $\mathcal{G}(X(\widetilde{\Omega G}))$, has the underlying groupoid 
\begin{equation} \label{StrictFromCrossed}
\xymatrix@M=5pt{ 
\widetilde{\Omega_{(0, \pi)} G} \rtimes_\alpha P_e G \ar@<5pt>[r]^-{s} \ar@<-5pt>[r]_-{t} & P_eG, \ar[l]|-{i} 
}
\end{equation}
where $i(\gamma) = (1, \gamma)$, $s(\Phi, \gamma) = \gamma$ and $t(\Phi, \gamma) = t(\Phi)\gamma$.

However, a more natural form for a strict Lie 2-group constructed from a central extension $\widetilde{\Omega G}$ would be
\begin{equation} \label{2GroupWithCentralExtension}
\xymatrix@M=5pt{ 
\widetilde{P_e G^{[2]}} \ar@<5pt>[r]^-{s} \ar@<-5pt>[r]_-{t} & P_eG, \ar[l]|{i} 
}
\end{equation}
i.e., its Lie group of morphisms is $\widetilde{P_e G^{[2]}} \subset \widetilde{\Omega G}$, and the maps $s$ and $t$ are as in  \cref{SourceTargetMaps}.
We claim that the missing ingredient to obtain such a form is the identity map $i$. It can be provided by a so-called fusion factorization, see \cite[Definition~5.5]{KristelWaldorf1}. 
 A \emph{fusion factorization} for a central extension $\widetilde{\Omega G}$ is a Lie group homomorphism
\begin{equation} \label{FusionFactorizationProperty}
 i: P_e G \to \widetilde{P_e G^{[2]}} \qquad \text{such that} \qquad \pi (i(\gamma)) = \gamma \cup \gamma.
\end{equation}

\begin{lemma}
\label{Lie-2-group-from-fusion-factorization}
Let $\widetilde{\Omega G}$ be a disjoint commutative central extension of $\Omega G$. Then, any fusion factorization $i$ for $\widetilde{\Omega G}$ provides an identity map  completing \cref{2GroupWithCentralExtension} to a strict Lie 2-group $\TwoGroupFusFac{\widetilde{\Omega G}}{i}$, together with a canonical isomorphism  $\TwoGroupFusFac{\widetilde{\Omega G}}{i} \cong \mathcal{G}(X(\widetilde{\Omega G}))$.
\end{lemma}

\begin{proof}
In order to show that $i$ turns \cref{2GroupWithCentralExtension} into a Lie 2-group, we use \cref{LemmaMinimalData2Group}:
The requirement that $\ker(s)$ and $\ker(t)$ commute is the assumption that $\widetilde{\Omega G}$ is disjoint commutative,
and the property \cref{FusionFactorizationProperty} implies that both  $t \circ i$ and $s \circ i$ are the identity on $P_eG$.

In order to construct the isomorphism
 $\TwoGroupFusFac{\widetilde{\Omega G}}{i} \cong \mathcal{G}(X(\widetilde{\Omega G}))$ we observe that 
 \begin{equation}
 \label{coincidence-of-crossed-modules}
X(\widetilde{\Omega G})=\mathcal{X}(\TwoGroupFusFac{\widetilde{\Omega G}}{i})
\end{equation}
where on the left is the crossed module of \cref{CanonicalCrossedModule} and $\mathcal{X}$ is the functor from the adjunction \cref{AdjunctionLieCross}.
Indeed, the crossed module on the left is $\widetilde{\Omega_{(0, \pi)} G} \to P_e G$ with the canonical action  given by \cref{CanonicalAction},
and the crossed module on the right is $\mathrm{ker}(s) \stackrel{t}{\to} P_eG$ with the action given by
\begin{equation} \label{AlphaViai}
  \alpha_\gamma(\Phi) = i(\gamma) \Phi\, i(\gamma)^{-1}.
\end{equation}
First, we recall from \cref{kernel-of-s} that $\mathrm{ker}(s)= \widetilde{\Omega_{(0, \pi)} G}$, and observe that the Lie group homomorphisms to $P_eG$ coincide. Second, for $\gamma \in P_e G$, the fusion factorization $i(\gamma)$ provides a concrete choice for a lift of $\gamma \cup \gamma$, which means that the formulas \cref{CanonicalAction} and \cref{AlphaViai} coincide. This shows the equality in \cref{coincidence-of-crossed-modules}.
Now, applying the functor $\mathcal{G}$ to \cref{coincidence-of-crossed-modules} and using the counit 
\begin{equation*}
\eta_{\TwoGroupFusFac{\widetilde{\Omega G}}{i}}:\TwoGroupFusFac{\widetilde{\Omega G}}{i} \to \mathcal{GX}(\TwoGroupFusFac{\widetilde{\Omega G}}{i})
\end{equation*}
 establishes the claimed isomorphism. 
 \end{proof}

Next we study existence and uniqueness of fusion factorizations.

\begin{lemma} \label{LemmaFusionFactorization}
Let $\widetilde{\Omega G}$ be a central extension of $\Omega G$.
If $G$ is semisimple, there exists at most one fusion factorization for $\widetilde{\Omega G}$.
\end{lemma}

\begin{proof}
Let $i$ and $i^\prime$ be two fusion factorizations.
We define a map $\varphi: P_e G \to \widetilde{\Omega G}$ by
\begin{equation*}
  \varphi(\gamma) = i(\gamma) i^\prime(\gamma)^{-1}.
\end{equation*}
As both $i(\gamma)$ and $i^\prime(\gamma)$ lie over $\gamma \cup \gamma$, $\varphi$ takes values in $\U(1) \subset \widetilde{\Omega G}$.
$\varphi$ is a group homomorphism, because
\begin{equation*}
  \varphi(\gamma_1)\varphi(\gamma_2) = i(\gamma_1) i^\prime(\gamma_1)^{-1} \varphi(\gamma_2) = i(\gamma_1)\varphi(\gamma_2) i^\prime(\gamma_1)^{-1} = \varphi(\gamma_1 \gamma_2).
\end{equation*}
By \cref{TheoremSemisimpleNoHoms}, $\varphi$ is trivial.
Hence $i = i^\prime$.
\end{proof}

\begin{remark}
The proof above shows that in the general (not necessarily semisimple) case, if a fusion factorization exists, then the Poincar\'e dual $(P_e G)^* = \Hom(P_e G, \U(1))$ acts freely and transitively on the set of fusion factorizations.
\end{remark}

We denote by $\sigma: \Omega G \to \Omega G$ the group homomorphism obtained by pullback with the \quot{flip} diffeomorphism $t \mapsto -t$.

\begin{lemma} \label{LemmaExistenceFusionFactorization}
Let $\widetilde{\Omega G}$ be a central extension of $\Omega G$.
Suppose there exists a group homomorphism $\tilde{\sigma} : \widetilde{\Omega G} \to \widetilde{\Omega G}$ covering $\sigma$ which is $\U(1)$-anti-equivariant in the sense that $\tilde{\sigma}(z\Phi) = \overline{z}\tilde\sigma(\Phi)$ for all $z \in \U(1)$ and $\Phi\in \widetilde{\Omega G}$.
Then there exists a unique fusion factorization $i$ such that $\tilde{\sigma} \circ i = i$.
\end{lemma}

\begin{proof}
Consider the map
\begin{equation*}
w : P_e G \times_{\Omega G} \widetilde{\Omega G} \longrightarrow \U(1), \qquad w(\gamma, \Phi) = \Phi^{-1} \tilde{\sigma}(\Phi),
\end{equation*}
where the fibre product is taken over the diagonal map $P_e G \to \Omega G$, $\gamma \mapsto \gamma \cup \gamma$.
Since $\pi(\Phi) = \pi(\tilde{\sigma}(\Phi)) = \gamma \cup \gamma$, we have $\pi(w(\gamma, \Phi)) = \const_e$; hence, $w(\gamma, \Phi)$ takes values in $\U(1)$. 
Moreover, $w$ is a group homomorphism:
\begin{equation*}
\begin{aligned}
w(\gamma, \Phi)w(\eta, \Psi) &= w(\gamma, \Phi) \Psi^{-1} \tilde{\sigma}(\Psi)\\
&= \Psi^{-1} w(\gamma, \Phi) \tilde{\sigma}(\Psi)\\
&= \Psi^{-1} \Phi^{-1} \tilde{\sigma}(\Phi)\tilde{\sigma}(\Psi) \\
&= w(\gamma\eta, \Phi\Psi).
\end{aligned}
\end{equation*}
For $z \in \U(1)$, we have
\begin{equation*}
  w(\gamma, z\Phi) = (z \Phi)^{-1} \tilde{\sigma}(z \Phi) = \overline{z}^2 \Phi^{-1} \tilde{\sigma}(z \Phi) = \overline{z}^2 w(\gamma, \Phi).
\end{equation*}
Hence, if $(\gamma, \Phi) \in \ker(w)$, then we have $(\gamma, z\Phi) \in \ker(w)$ if and only if $\overline{z}^2 = 1$, that is, $z = \pm 1$.
We obtain that $\pr_1:\ker(w) \to P_eG$ is a double cover.
Since $P_e G$ is contractible, this double cover is necessarily trivial.
Therefore, its restriction  to the identity component $\ker(w)_0$ is an isomorphism of Lie groups $\pr_1{|_{\ker(w)_0}}:\ker(w)_0 \to P_eG$.
Then, $i := \pr_2 \circ (\pr_1{|_{\ker(w)_0}})^{-1}$ is a fusion factorization.

Conversely, any fusion factorization $i$ such that $\widetilde{\sigma} \circ i = i$ gives a section of $\pr_1:\ker(w) \to P_eG$ with $i(\const_e) = 1$. 
But since the fibres of $\ker(w)$ are discrete, there is at most one such section.
\end{proof}

\begin{theorem} \label{ThmExistenceFusionFactorization}
Let $\widetilde{\Omega G}$ be a central extension of $\Omega G$, where $G$ is simply connected and semisimple.
Then, there exists a unique fusion factorization for $\widetilde{\Omega G}$.
\end{theorem}

\begin{proof}
Uniqueness was shown in \cref{LemmaFusionFactorization}, so it remains to show existence.
We claim that  our assumptions on $G$ imply the conditions of \cref{LemmaExistenceFusionFactorization}.
To see this, consider the dual (inverse) central extension $\widetilde{\Omega G}{}^*$.
Then $\sigma^*\widetilde{\Omega G}{}^*$ is another central extension, which comes with a canonical Lie group homomorphism 
\begin{equation*}
\tilde{\sigma}^\prime : \widetilde{\Omega G} \to \sigma^*\widetilde{\Omega G}{}^*
\end{equation*}
that covers $\sigma$ and is $\U(1)$-anti-equivariant.
%
%
By our assumptions, the homomorphism $h_0(\CExt(\Omega G)) \longrightarrow H^2_c(L\mathfrak{g}, \R)$ of \cref{CanonicalMap} is injective, so that central extensions are determined their 2-cocycles.
Now, if $\omega$ is the 2-cocycle classifying $\widetilde{\Omega G}$, then the dual extension $\widetilde{\Omega G}{}^*$ is classified by $-\omega$.
By \cref{lemma:equivariantcocycle} we may assume that $\omega$ is $G$-equivariant, hence of the form \cref{FormulaForCocycle}. For such a cocycle $\omega$ the action of $\sigma$ on $H^2_c(L\mathfrak{g}, \R)$ replaces $\omega$ by $-\omega$, so that $\sigma^*\widetilde{\Omega G}{}^*$ is again classified by $\omega$.
By \cref{LemmaAutoCExt}, $\sigma^*\widetilde{\Omega G}{}^*$ is, as a central extension,  isomorphic to $\widetilde{\Omega G}$. 
The post-composition of this isomorphism with $\tilde{\sigma}^\prime$ provides an anti-linear bundle map $\tilde{\sigma}$ covering $\sigma$, and \cref{LemmaExistenceFusionFactorization} completes the proof.
\end{proof}

\begin{remark}
Observe that the proof of \cref{ThmExistenceFusionFactorization} actually shows that under the assumptions of \cref{ThmExistenceFusionFactorization}, there exists a map $\tilde{\sigma}$ as in \cref{LemmaExistenceFusionFactorization}, and the unique fusion factorization $i$ satisfies additionally $\tilde{\sigma} \circ i = i$.
\end{remark}

\subsection{Classification of the  Lie 2-groups}

\label{classification-of-Lie-2-groups}

In this section we prove that -- in case of a simple and simply connected Lie group $G$ and for a \quot{basic} central extension -- our canonical Lie 2-group $\mathcal{G}=\mathcal{G}(X(\widetilde{\Omega G}))$ of \cref{SectionCrossedModulesFromCExt} becomes under geometric realization a 3-connected cover of $G$. For this purpose we will   use the methods developed in \cite{BaezStevenson,BCSS}. 
We start by recalling some notions and basic facts about Lie 2-groups (as used, e.g., in \cite[\S4.2]{BCSS}).
A \emph{strict homomorphism} between strict Lie 2-groups  consists of two Lie group homomorphisms (one between the morphism groups and one between the object groups), which intertwine all structure maps.
The \emph{strict kernel} of such a strict homomorphism is the 2-group obtained by taking the level-wise kernels.
It is a Lie 2-group if both kernels are submanifolds (which is automatic in the finite-dimensional case).
A sequence
\begin{equation*}
0 \longrightarrow \mathcal{K} \longrightarrow \mathcal{G} \longrightarrow \mathcal{H} \longrightarrow 0
\end{equation*}
of strict Lie 2-groups and strict homomorphisms is called \emph{strictly exact} if it is exact on both object and morphism level.

Taking the nerve of a strict Lie 2-group $\mathcal{G}$ and forgetting the smooth structure, we obtain a simplicial space $N \mathcal{G}$, where $(N\mathcal{G})_0 = \mathrm{Ob}(\mathcal{G})$ and whose $n$-th space, $n \geq 1$, is the space of  $n$-strings of composable morphisms, 
\begin{equation*}
  (N\mathcal{G})_n = \{ (x_1, \dots, x_n) \in \mathrm{Mor}(\mathcal{G})^n \mid s(x_j) = t(x_{j-1}), j=2, \dots, n\}.
\end{equation*}
Applying the geometric realization functor, we obtain a CW complex $|\mathcal{G}|$,  the \emph{geometric realization} of $\mathcal{G}$.
Pointwise multiplication in $\mathrm{Mor}(\mathcal{G})$ endows each of the spaces $(N \mathcal{G})_n$ with the structure of a topological group (in fact, a Lie group) for which the simplicial structure maps are homomorphisms.
Put differently, we have a group object in the category of simplicial spaces, and since the geometric realization functor preserves finite products, it sends group objects to group objects, so that $|\mathcal{G}|$ acquires the structure of a topological group (see also Lemma~1 in \cite{BaezStevenson}).
It is moreover a fact that geometric realization takes a short strictly exact sequence of Lie 2-groups to an exact sequence of topological groups \cite[\S4.2]{BCSS}.

\medskip

Let $G$ be a finite-dimensional, connected, and semisimple Lie group and let $\widetilde{\Omega G}$ be a disjoint commutative central extension of the loop group $\Omega G$. 
Let 
\begin{equation*}
\mathcal{G} = \mathcal{G}(X(\widetilde{\Omega G}))
\end{equation*}
be the  Lie 2-group corresponding to the crossed module constructed in \cref{SectionCrossedModulesFromCExt}.
Since two objects $\beta_1,\beta_2\in P_eG$ are isomorphic in $\mathcal{G}$ if and only if they have the same end point, $\mathcal{G}$ comes with a canonical strict morphism to the strict Lie 2-group $G\dis$  (see \cref{ex:2-groupsOrdinaryGroups}), given by end point evaluation.
Let $G{\dis'}$ be the strict 2-group with objects $P_e G$, morphisms $P_e G^{[2]}$, and the obvious structure maps.
Then we have a factorization,
\begin{equation}
\label{Factorization}
  \mathcal{G} \longrightarrow G\dis^\prime \longrightarrow G\dis,
\end{equation}
where the first map is the identity on objects and the footpoint projection on morphisms, while the morphism $G\dis^\prime \to G\dis$ is end point evaluation, both on objects and morphisms. It is straightforward to show that the  second arrow in \cref{Factorization} is a weak equivalence.
%
By construction, the strict kernel of the first homomorphism in \cref{Factorization} is the trivial group on objects and $\U(1) \subset \widetilde{\Omega G}$ on morphisms; in other words, it is the strict 2-group $B\U(1)$ (\cref{ex:2groupsAbelian}).
We therefore get a strict short exact sequence of Lie 2-groups
\begin{equation}
\label{ShortExactSequence}
B\U(1) \longrightarrow \mathcal{G} \longrightarrow G^\prime\dis.
\end{equation}
Geometric realization takes the short exact sequence of Lie 2-groups to a short exact sequence of topological groups.
Here we have $|B\U(1)| \simeq K(\Z, 2)$, and  $|G^\prime_{\mathrm{dis}}| \simeq |G_{\mathrm{dis}}| \simeq G$, so we obtain a homotopy fibre sequence
\begin{equation*}
  K(\Z, 2) \longrightarrow |\mathcal{G}| \longrightarrow G.
\end{equation*}
As $\pi_k(K(\Z, 2)) = 0$ for $k \neq 2$, we obtain that it induces an isomorphism on $\pi_k$ for all $k \notin \{2, 3\}$.
In the latter range, we obtain the exact sequence
\begin{equation}
\label{ExactSequence}
 0 \longrightarrow \pi_3(|\mathcal{G}|) \longrightarrow \pi_3(G) \stackrel{\varphi}{\longrightarrow} \Z \longrightarrow \pi_2(|\mathcal{G}|) \longrightarrow 0,
\end{equation}
where clearly it is crucial to understand the connecting homomorphism $\varphi$.
Let $[\overline{\omega}] \in H^2(\Omega G, \Z)$ be the class corresponding to the central extension $\widetilde{\Omega G}$.

\begin{lemma}
The homomorphism $\varphi$ in \cref{ExactSequence} is given by
  \begin{equation*}
     \varphi(f) = \langle \hat{f}^*\overline{\omega}, [S^2]\rangle,
  \end{equation*}
  where $\hat{f} \in \pi_2(\Omega G)$ is the image of $f$ under the isomorphism $\pi_3(G) \cong \pi_2(\Omega G)$.
\end{lemma}

\begin{proof}
We use the following construction of \cite{BaezStevenson}, see Lemma~1 and \S5.3:
Let $\mathcal{G}$ be a Lie 2-group and $(H, \mathcal{G}_0, \alpha, t)$ is the corresponding crossed module.
Then there exists a weakly contractible topological group $EH$ containing $H$ as a normal subgroup, together with an action of $\mathcal{G}_0$ on $EH$ extending the action of $\mathcal{G}_0$ on $H$.
Moreover, $H$ is embedded as a normal subgroup of the semidirect product $EH \rtimes \mathcal{G}_0$, and we have a short exact sequence of topological groups 
\begin{equation*}
 H \longrightarrow EH \rtimes \mathcal{G}_0 \longrightarrow |\mathcal{G}|,
\end{equation*}
whitnessing $\mathcal{G}$ as the quotient
\begin{equation*}
  |\mathcal{G}| \cong (EH \rtimes \mathcal{G}_0)/H.
\end{equation*}
These constructions are functorial in $\mathcal{G}$, so we can apply it to the strict short exact sequence \cref{ShortExactSequence}.
The object group $\mathcal{G}_0$ is contractible in each case (being either trivial or the path group $P_e G$), hence the geometric realization is isomorphic to $BH$ in each case.
Identifying $\Omega_{(0, \pi)} G \cong \Omega G$ and $\widetilde{\Omega_{(0, \pi)} G} \cong \widetilde{\Omega G}$ (see \cref{LemmaRestrictionIso}), we obtain that under geometric realization, the strict short exact sequence \cref{ShortExactSequence} corresponds to the short exact sequence of topological groups
\begin{equation*}
 B\U(1) \longrightarrow B\widetilde{\Omega G} \longrightarrow B{\Omega G}.
\end{equation*}
An inspection of the construction in \cite[\S5.3]{BaezStevenson} reveals that, as expected, this sequence is just the one obtained from applying the classifying space functor $B$ to the short exact sequence corresponding to the central extension $\widetilde{\Omega G}$.

It is now a general fact that for a principal $\U(1)$-bundle $\U(1) \to P \to B$, the boundary map $\pi_2(B) \to \pi_1(\U(1)) \cong \Z$ of the corresponding long exact sequence of homotopy groups is the map that sends $[f] \in \pi_2(B)$ to the first Chern number $\langle c_1(f^*P), [S^2]\rangle$ of the bundle $f^* P \to S^2$.
In our case, the first Chern class of $\widetilde{\Omega G}$ is represented by the left-invariant 2-form $\overline{\omega}$, and so the result follows.
\end{proof}

We summarize the results of this section as the following theorem.

\begin{theorem}
\label{th:classification}
Let $G$ be a simple Lie group, and let $\widetilde{\Omega G}$ be a basic central extension of $\Omega G$, i.e., one whose classifying cocycle  $\overline{\omega}$ is a generator of $H^2(\Omega \mathfrak{g}, \Z) \cong \Z$. Let 
$\mathcal{G} = \mathcal{G}(X(\widetilde{\Omega G}))$
be the  Lie 2-group corresponding to the crossed module constructed in \cref{SectionCrossedModulesFromCExt}.
Then,  $\pi_3(|\mathcal{G}|) = \pi_2(|\mathcal{G}|) = 0$.
In particular, if $G$ is simple and simply connected, then $|\mathcal{G}|$ is the 3-connected cover of $G$.
\end{theorem}

\section{Comparison with other constructions}

In this section we carry out the comparison between our constructions of  \cref{section-3} and the constructions of Baez et al.\ and the second-named author. 

\subsection{The BCSS string 2-group}

\label{Comparison-with-BCSS}

We start be reviewing the main construction of Baez et al.\ \cite[Prop. 25]{BCSS}. 
We remark that their construction is presented as if it results into as that of a Fr\'echet Lie 2-group, but in fact it results into a crossed module of Fr\'echet Lie groups, to which then the functor $\mathcal{G}$ from \cref{AdjunctionLieCross} is applied without mention. 
So we better describe that crossed module directly.  

Let $P_eG\BCSS \subset C^{\infty}([0,  2\pi], G)$ be the Fr\'echet submanifold of paths starting at $e\in G$. Note that -- in contrast to our setting -- there is no flatness assumption; moreover,  paths are parameterized by $[0, 2\pi]$ instead of $[0, \pi]$. 
We denote by $\Omega G\BCSS \subset P_eG\BCSS$ the Fr\'echet manifold of closed paths,  and assume that 
\begin{equation*}
1 \to \U(1) \to \widetilde{\Omega G}\BCSS \to \Omega G\BCSS \to 1
\end{equation*}
is a central extension.
A Lie group homomorphism 
\begin{equation*}
t\BCSS:\widetilde{\Omega G}\BCSS \to P_eG\BCSS
\end{equation*}
is defined by projection and inclusion. Under certain assumptions on the central extension,  including the condition that $G$ is of Cartan type and classified by a level $k\in \Z$,  a central crossed module action
\begin{equation*}
\alpha\BCSS: P_eG\BCSS \times  \widetilde{\Omega G}\BCSS \to \widetilde{\Omega G}\BCSS
\end{equation*} 
can be defined (in a difficult way,  using Lie-algebraic methods). It  will not be necessary to review this construction here,  as we will prove below that it restricts to our canonical action. We denote the crossed module defined this way by $X\BCSS(G, k)$; it is precisely the one described in \cite[Prop.\ 25]{BCSS}.

In the following we will show that $X\BCSS(G, k)$ is weakly equivalent to our canonical crossed module $X(\widetilde{\Omega G})$ from \cref{CanonicalCrossedModule}.
In order to do so,  we first have to specify the disjoint commutative central extension $\widetilde{\Omega G}$ required there. 
We consider the maps
\begin{equation*}
\xymatrix{\Omega_{(0, \pi)}G \ar[r]^-{\res}  & P_eG \ar[r]^-{\rep} & P_eG\BCSS}
\end{equation*}
defined by $\rep(\gamma)(x) := \gamma(\frac{1}{2}x)$,  for $\gamma\in P_eG$, $x\in [0, 2\pi]$,  and $\res(\eta)(x) := \eta(x)$,  for $\eta\in \Omega_{(0, \pi)}G$ and $x\in [0, \pi]$. Note that $\rep$ and $\res$ are Lie group homomorphisms. Their composition will be denoted by $r := \rep\circ \res$. We let 
\begin{equation*}
\widetilde{\Omega_{(0, \pi)} G} := r^{*}\widetilde{\Omega G}\BCSS
\end{equation*}
be the pullback central extension. By \cref{LemmaRestrictionIso},  this is the restriction of a central extension $\widetilde{\Omega G}$,  as required. 
Note that $\widetilde{\Omega G}$ is disjoint commutative since $G$ is semisimple and simply connected,  due to \cref{CorollarySemisimpleSimplyConnectedDisjointCommutative}.
We obtain -- by construction -- a commutative diagram:
\begin{equation*}
\xymatrix{\widetilde{\Omega_{(0, \pi)} G} \ar@/_3pc/[dd]_{t} \ar[d]_{p} \ar[r]^-{\tilde r} & \widetilde{\Omega G}\BCSS \ar@/^3pc/[dd]^{t\BCSS} \ar[d]\\ \Omega_{(0, \pi) G} \ar[d]_{\res} \ar[r]_{r} & \Omega G\BCSS \ar@{^(->}[d] \\ P_eG \ar[r]_-{\rep} & P_eG\BCSS }
\end{equation*}

\begin{lemma}
\label{stricthomomorphismR}
The maps $\tilde r$ and $\rep$ constitute a strict homomorphism 
\begin{equation*}
R:X(\widetilde{\Omega G}) \longrightarrow X\BCSS(G, k)
\end{equation*}
of crossed modules. 
\end{lemma}

\begin{proof}
Since the diagram is commutative,  it remains to prove that the crossed module actions are exchanged,  i.e.,  that
\begin{equation}
\label{coincidence-of-actions}
\alpha\BCSS_{\rep(\gamma)}(\tilde r(\Phi))=\tilde r(\alpha_{\gamma}(\Phi))
\end{equation}
for all $\gamma\in P_eG$ and $\Phi \in \widetilde{\Omega_{(0, \pi)} G}$. We note that
\begin{align*}
t\BCSS(\alpha\BCSS_{\rep(\gamma)}(\tilde r(\Phi))) &=\rep(\gamma)\cdot t\BCSS(\tilde r(\Phi))) \cdot \rep(\gamma)^{-1}
\\ &=\rep(\gamma)\cdot \rep(t(\Phi)) \cdot \rep(\gamma)^{-1}
\\ &=\rep(\gamma\cdot t(\Phi) \cdot \gamma^{-1})
\\ &= r(\eta(\gamma, \Phi))\text{, } 
\end{align*}
where $\eta(\gamma, \Phi) :=(\gamma\cdot t(\Phi) \cdot \gamma^{-1}) \cup \const_e \in \Omega_{(0, \pi)}G$. This shows that we obtain a well-defined element
\begin{equation*}
\alpha_{\gamma}(\Phi):=(\eta(\gamma, \Phi), \alpha\BCSS_{\rep(\gamma)}( \tilde r(\Phi)))\in\widetilde{\Omega_{(0, \pi)} G}\text{.}
\end{equation*}
The map $\alpha_{\gamma}$ defined like this is a smooth,  central crossed module action for $t: \widetilde{\Omega G} \to P_eG$; moreover,  by construction,  it satisfies \cref{coincidence-of-actions}. Since $G$ is semisimple,  it coincides with our canonical action  by \cref{ThmUniqueAction}. 
\end{proof}

We may thus say that our canonical action $\alpha$ is the restriction of the action $\alpha\BCSS$ along the homomorphism $R$. 

\begin{theorem}
\label{equivalence-to-BCSS}
The homomorphism $R$ of \cref{stricthomomorphismR} establishes a weak equivalence of  crossed modules of Fr\'echet Lie groups, 
\begin{equation*}
X(\widetilde{\Omega G}) \cong X\BCSS(G, k).
\end{equation*}
\end{theorem}  

\begin{proof}
Every strict homomorphism between crossed modules determines a weak homomorphism,  a.k.a.\ a butterfly,  see \cite[\S 4.5]{Aldrovandi2009}. In the case of $R$,  this butterfly is
\begin{equation*}
\xymatrix{\widetilde{\Omega_{(0, \pi)}G} \ar[dr]^{\kappa}\ar[dd]_{t} && \widetilde{\Omega G}\BCSS \ar[dl]\ar[dd]^{t\BCSS} \\ & \widetilde{\Omega G}\BCSS \!\!\!\!\! \rtimes P_eG \ar[dl]\ar[dr]^{j} \\ P_eG &&  P_eG\BCSS }
\end{equation*}
where the group in the middle is the semi-direct product w.r.t.\ the action $\alpha\BCSS$ induced along $\rep: P_eG \to P_eG\BCSS$,  
and the NE-SW-sequence is the corresponding split extension. Moreover,  
\begin{equation*}
\begin{aligned}
\kappa(\Phi) &:=(\tilde r(\Phi)^{-1},  t(\Phi)) \qquad\text{and}\qquad  \\
  j(\Phi, \gamma) &:= \rep(\gamma)\cdot t\BCSS(\Phi).
  \end{aligned}
  \end{equation*}
By \cite[\S 5.2]{Aldrovandi2009},  a butterfly establishes a weak equivalence if it is reversible,  meaning that its NW-SE-sequence
\begin{equation*}
\widetilde{\Omega_{(0, \pi)}G} \stackrel\kappa\longrightarrow \widetilde{\Omega G}\BCSS \!\!\!\!\! \rtimes P_eG \stackrel j\longrightarrow P_eG\BCSS
\end{equation*}
is also short exact. Since that sequence is always a complex (for any butterfly),  it remains to prove that it is an exact sequence of groups and a locally trivial principal bundle.

Since $r$ is injective,  the covering map $\tilde r$ is also injective,  and hence $\kappa$ is injective. In order to show the surjectivity of $j$,  we consider  $\gamma\in P_eG^{\BCSS}$ and choose a smooth map $\varphi: [0, 2\pi] \to [0, \pi]$ with $\varphi(0)=0$ and $\varphi(2\pi)=\pi$ that is flat at its end points. 
Then, for any lift $\Phi\in \widetilde{\Omega G}\BCSS$ of $\gamma\cdot \rep(\gamma\circ \varphi)^{-1} \in \Omega G\BCSS$, we have $j(\Phi, \gamma\circ \varphi)=\gamma$, hence $j$ is surjective.
The fact that $\varphi$ can be chosen to be the same for all $\gamma\in P_eG\BCSS$ and the fact that $\Phi$ can be chosen in a locally smooth way shows that $j$ has local sections,  and hence is a principal bundle. 

It remains to show exactness in the middle. Let $(\Phi, \gamma) \in \widetilde{\Omega G}\BCSS \rtimes P_eG$ be in the kernel of $j$,  i.e.,  $\rep(\gamma)\cdot t\BCSS(\Phi)=\const_e$. 
Then 
\begin{equation*}
(\gamma \cup \const_e, \Phi^{-1}) \in \widetilde{\Omega_{(0, \pi)} G}
\end{equation*}
 is sent to $(\Phi, \gamma)$ under $\kappa$.
\end{proof}

\subsection{The diffeological string 2-group}

\label{diffeological-model}

The following construction of a diffeological 2-group is implicit in \cite{WaldorfStringGroupModels,WaldorfString, WaldorfTransgressive},  but has not been described explicitly.  It takes as input data  a \emph{fusion extension},  i.e.   central extension 
\begin{equation}
        \label{eq:bce}
        1\to\U(1) \to \widetilde{LG} \to LG\to 1
\end{equation}
of Fr\'echet Lie groups that is equipped with a \emph{multiplicative fusion product}.

In the following we use without further notice the fully faithful functor from Fr\'echet manifolds to diffeological spaces in order to embed everything into that setting.
We let $P_eG\si$ be the diffeological space of paths in $G$ with sitting instants (constant in neighborhoods of its end points) starting at $e\in G$,  and by $P_eG\si ^{[k]}$ its $k$-fold fibre products along the endpoint evaluation $\ev: P_eG\si \to G$. As before,  we have a smooth map $\cup: P_eG\si^{[2]} \to LG$. A \emph{fusion product} is a bundle morphism
\begin{equation*}
\lambda: \pr_{12}^{*}\cup^{*}\widetilde{LG} \otimes \pr_{23}^{*}\cup^{*}\widetilde{LG} 
\longrightarrow
 \pr_{13}^{*}\cup^{*}\widetilde{LG}
\end{equation*}  
over $P_eG\si^{[2]}$ that satisfies the evident associativity condition over $P_eG\si^{[4]}$. Moreover,  it is called \emph{multiplicative} if it is a group homomorphism,  see \cite{WaldorfStringGroupModels,WaldorfString, WaldorfTransgressive} for more details.

\begin{remark}
Fusion extensions may -- on first view -- look odd and involved,  but in fact  appear very naturally.
Indeed,  there are at least the following three ways to obtain a fusion extension of the loop group $LG$ of a Lie group $G$:
\begin{enumerate}[(1)]

        \item 
                Transgression of any multiplicative bundle gerbe over $G$ results in a fusion extension of $LG$; this is explained in  \cite[\S 2]{WaldorfString}. 

        \item
                The \emph{Mickelsson model} produces a canonical fusion extension for any simply connected Lie group $G$; this is explained  in \cite[Example 2.6]{WaldorfTransgressive}. 

        \item
                The operator-algebraic  \emph{implementer model} \cite{KristelWaldorf1} produces a canonical fusion extension for $L\Spin(d)$.

\end{enumerate}
\end{remark}

We note that every fusion extension comes equipped with a  \emph{fusion factorization},   uniquely characterized by the property that is neutral with respect to fusion \cite[Prop. 3.1.1]{WaldorfTransgressive}. 
The following result,  which is nothing but a reformulation of the given conditions,  constructs from a fusion extension  a strict diffeological 2-group.

\begin{proposition}\label{lem:2groupFromFusionExtension}
Given a fusion extension as above,  the following structure yields a central strict diffeological 2-group $\mathcal{S}(\widetilde{LG}, \lambda)$:

        \begin{itemize}

                \item 
                        The diffeological group of objects is $P_eG\si$.

                \item
                        The diffeological group of morphisms is 
                        \begin{equation*}
                                \widetilde{\Omega G}{}\diff := P_eG\si^{[2]} \times_{LG} \widetilde{LG}\text{, }
                        \end{equation*}
                        where the fibre product is taken along the map $\cup: P_eG\si^{[2]} \to LG$.

                        \item
                        Source and target maps are $s(\gamma_1, \gamma_2, \Phi) := \gamma_2$ and $t(\gamma_1, \gamma_2, \Phi) := \gamma_1$.

                \item
                        Composition is the fusion product $\lambda$ of $\widetilde{LG}$:
                        \begin{equation*}
                                (\gamma_0, \gamma_1, \Phi')\circ (\gamma_1, \gamma_2, \Phi) := (\gamma_0, \gamma_2, \lambda(\Phi' \otimes \Phi))\text{.}
                        \end{equation*}

                \item
                        The identity morphism of $\gamma\in P_eG\si$ is $(\gamma, \gamma, i(\gamma))$,  where $i$ is the fusion factorization associated to $\lambda$.

        \end{itemize}
\end{proposition}

\begin{remark}
It is easy to check that $\pi_1\mathcal{S}(\widetilde{LG}, \lambda) =\U(1)$ and $\pi_0\mathcal{S}(\widetilde{LG}, \lambda) \cong G$,  so that $\mathcal{S}(\widetilde{LG}, \lambda)$ is a diffeological Lie 2-group extension
\begin{equation*}
B\U(1) \longrightarrow \mathcal{S}(\widetilde{LG}, \lambda)\longrightarrow G\dis\text{.}
\end{equation*}  
\end{remark}

\begin{remark}
As noticed in \cite[\S 5.2]{KristelWaldorf1} and  deduced in general in \cref{SectionStrict2GroupsCrossedModules},  the fusion product $\lambda$ is  already determined by its fusion factorization $i$; moreover,  the subgroups 
\begin{equation*}
\widetilde{\Omega_{(0, \pi)}G}{}\diff =\mathrm{ker}(s) \subset \widetilde{\Omega G}{}\diff \qquad \text{and} \qquad \widetilde{\Omega_{(\pi, 2\pi)}G}{}\diff =\mathrm{ker}(t) \subset \widetilde{\Omega G}{}\diff
\end{equation*}
 commute with each other. 
\end{remark}

The goal of this section is to compare the diffeological Lie 2-group $\mathcal{S}(\widetilde{LG}, \lambda)$ with our constructions from \cref{section-3},  and it is best to do this on the level of crossed modules. The diffeological crossed module $\mathcal{X}(\mathcal{S}(\widetilde{LG}, \lambda))$ is
\begin{equation*}
t:\widetilde{\Omega_{(0, \pi)}G}{}\diff \to P_eG\si\text{, }
\end{equation*}
with the central crossed module action
\begin{equation*}
\alpha\diff: P_eG\si \times\widetilde{\Omega_{(0, \pi)}G}{}\diff \to \widetilde{\Omega_{(0, \pi)}G}{}\diff 
\end{equation*}  
given by $\alpha\diff(\gamma, \Phi):= i(\gamma)\cdot\Phi\cdot i(\gamma)^{-1}$. 

\begin{remark}
As in \cref{SectionCrossedModulesFromCExt},  we can observe here immediately that this action does not even depend on the fusion factorization,  and hence,  that the crossed module $\mathcal{X}(\mathcal{S}(\widetilde{LG}, \lambda))$ is completely independent of the fusion product $\lambda$. 
However,  the condition that the subgroups $\widetilde{L_{(0, \pi)}G}$ and $\widetilde{L_{(\pi, 2\pi)}}G$ commute has to be imposed (it is slightly weaker than disjoint commutativity). 
\end{remark}

In order to explore the relation between the diffeological crossed module $\mathcal{X}(\mathcal{S}(\widetilde{LG}, \lambda))$ and our  crossed module $X(\widetilde{\Omega G})$ from \cref{SectionCrossedModulesFromCExt},  we assume that $\widetilde{LG}$ is a disjoint commutative central extension of a Lie group $G$; then,  both crossed modules are defined. 
We obtain a commutative diagram
\begin{equation*}
\xymatrix{ \widetilde{\Omega_{(0, \pi)}G}{}\diff\ar@{^(->}[r] \ar[d]_{t} & \widetilde {\Omega_{(0, \pi)}G} \ar[d]^{t} \\ P_eG\si \ar@{^(->}[r]  & P_eG }
\end{equation*}
whose horizontal arrows are inclusions (paths with sitting instants are flat). Moreover,  we observe that the action $\alpha\diff$ and or canonical action $\alpha$ are defined in exactly the same way. Hence,  above diagram constitutes a strict homomorphism of diffeological crossed modules
\begin{equation}
\label{inclusion-of-diffeological-model}
\mathcal{X}(\mathcal{S}(\widetilde{LG}, \lambda)) \to X(\widetilde{\Omega G})\text{.} 
\end{equation}

\begin{theorem}
\label{weak-equivalence-diffeological}
The homomorphism \cref{inclusion-of-diffeological-model} is a weak equivalence, 
\begin{equation*}
\mathcal{X}(\mathcal{S}(\widetilde{LG}, \lambda)) \cong X(\widetilde{\Omega G})\text{.}
\end{equation*}
In particular,  there is a canonical weak equivalences of diffeological 2-groups
\begin{equation*}
\mathcal{S}(\widetilde{LG}, \lambda) \cong \mathcal{G}(X(\widetilde{\Omega G}))\cong \mathcal{G}(\widetilde{\Omega G}, i)\text{.}
\end{equation*} 
\end{theorem}

\begin{proof}
We proceed as in the proof of \cref{equivalence-to-BCSS} and consider the butterfly
\begin{equation*}
\xymatrix{\widetilde{\Omega_{(0, \pi)}G}{}\diff \ar[dr]^{\kappa}\ar[dd]_{t} && \widetilde{\Omega_{(0, \pi)} G} \ar[dl]\ar[dd]^{t} \\ & \widetilde{\Omega_{(0, \pi)} G} \rtimes P_eG\si \ar[dl]\ar[dr]^{j} \\ P_eG\si &&  P_eG\text{, } }
\end{equation*}
where now $\kappa(\Phi):=(\Phi^{-1},  t(\Phi))$  and $j(\Phi, \gamma):= \gamma t(\Phi)$. We use again \cite[\S 5.2]{Aldrovandi2009} and have to prove that the NW-SE-sequence is short exact.
The proofs that $\kappa$ is injective and that $j$ is surjective and has local sections go as for \cref{equivalence-to-BCSS}. 
For exactness in the middle,  we observe that an equality $\gamma t(\Phi)=\const_e$ implies that $t(\Phi)$ has sitting instants,  and hence $\Phi \in \widetilde{\Omega_{(0, \pi)}G}{}\diff$.   
 \end{proof}

\addcontentsline{toc}{section}{\refname}

\bibliography{literature}
\bibliographystyle{kobib}

\end{document}